\documentclass[12pt, notitlepage]{amsart}
\usepackage{latexsym, amsfonts, amsmath, amssymb, amsthm, cite}
\usepackage{graphicx}
\usepackage[linktocpage=true,colorlinks,citecolor=blue,linkcolor=blue,urlcolor=blue]{hyperref}
\usepackage{amssymb}
\usepackage{cite}
\usepackage{bbm}
\usepackage{amsmath}
\usepackage{latexsym}
\usepackage{amscd}
\usepackage{tikz}
\usepackage{amsthm}
\usepackage{mathrsfs}
\usepackage{url}
\usepackage{bbm}
\usepackage[utf8]{inputenc}
\usepackage[english]{babel}
\usepackage{amsfonts}
\usepackage{mathtools}
\usepackage{comment}
\usepackage[autostyle]{csquotes}
\usepackage[colorinlistoftodos]{todonotes}
\usepackage[noabbrev,capitalize]{cleveref}
\usepackage{adjustbox}
\usepackage{mathrsfs}
\crefname{equation}{}{}
\usepackage{graphicx, tikz}
\usepackage[margin=1in]{geometry}
\usepackage{microtype}

\numberwithin{equation}{section}
\def\R{\mathbb R}
\def\Z{\mathbb Z}
\def\C{\mathbb C}

\def\E{\mathbb E}

\def\CF{\mathcal F}

\def\wt{\widetilde}
\def\gcd{\operatorname{gcd}}

\newcommand{\eps}{\varepsilon}
\newcommand{\str}{\operatorname{str}}
\newcommand{\sml}{\operatorname{sml}}
\newcommand{\unf}{\operatorname{unf}}
\newcommand{\vol}{\operatorname{vol}}
\newcommand{\lcm}{\operatorname{lcm}}
\newcommand{\Mod}[1]{\ (\mathrm{mod}\ #1)}

\newtheorem{theorem}{Theorem}[section]
\newtheorem{lemma}[theorem]{Lemma}
\newtheorem{proposition}[theorem]{Proposition}

\newtheorem{conjecture}[theorem]{Conjecture}

\theoremstyle{remark}
\newtheorem{remark}[theorem]{Remark}

\theoremstyle{definition}
\newtheorem{definition}[theorem]{Definition}

\theoremstyle{remark}

\numberwithin{equation}{section}

\begin{document}
\title{Maximising the number of solutions to linear equations}
\author[Ran\dj elovi\'c]{\v{Z}arko Ran\dj elovi\'c}
\address{Mathematical Institute of the Serbian Academy of Sciences and Arts
Kneza Mihaila 36
Belgrade 11000, Serbia}
\email{zarko.randjelovic@turing.mi.sanu.ac.rs}

\author[Shao]{Xuancheng Shao} 
\address{Department of Mathematics, University of Kentucky, Lexington, KY 40506, USA}
\email{xuancheng.shao@uky.edu}

\author[Xu]{Max Wenqiang Xu}
\address{Yau Mathematical Sciences Center, Tsinghua University, Beijing, 100084, China}
\address{Beijing Institute of Mathematical Sciences and Applications, Beijing, 101408, China}
\email{maxxu1729@gmail.com, maxxu@tsinghua.edu.cn}

\author[Zhang]{Shengtong Zhang}
\address{Department of Mathematics, Stanford University, Stanford, CA.}
\email{stzh1555@stanford.edu}

\author[Zhou]{Yuan Zhou}
\address{Department of Mathematics, University of Kentucky, Lexington, KY 40506, USA}
\email{yuan.zhou@uky.edu}

\begin{abstract}
    We study the asymptotically maximal possible number of integer solutions to the linear equation $ax+by+cz = 0$ with a fixed choice of $a, b, c \in \mathbb{Z}$ and variables $x, y, z \in S$ for some finite set $S\subset \mathbb{Z}$, as $|S|\to +\infty$. Define $\gamma_{a, b, c}$ to be the largest constant for which there are arbitrary large finite sets $S\subset \mathbb{Z}$ such that the number of solutions to $ax+by+cz=0$ with $x,y,z\in S$ is $\gamma_{a,b,c}|S|^2-o(|S|^2)$. We prove structural results for general $a, b, c$ and moreover, we show that $5/13\le \gamma_{1,1,-3}\le 1/2-\delta$ for some constant $\delta>0$. In addition we show that the limit as $a \rightarrow \infty$ of $\gamma_{1,1,-a}$ is equal to precisely $1/5$. 
    
\end{abstract}

\maketitle
\section{Introduction}
Counting solutions to a given linear equation $\sum_{1\le i \le d} a_i x_i = 0$, where each variable $x_i$ is in a given set $S$, is a simply stated question. In this paper, we study the problem of finding an optimal $S$ such that the number of solutions is maximized. In particular we focus on the case with three variables.  

We count the solutions to the equation $$ax+by +cz = 0$$ with $a, b, c \in \mathbb{Z}$ and $x, y, z \in S$ for some finite set $S\subset \Z$. Our goal is to understand what is the maximal number of solutions one can have among all sets $S$ with a given size $|S|=N$ and with $a,b,c$ fixed. We define 
$$
T_{a,b,c}(S) :=   \#\{(x,y,z)\in S^3:ax+by+cz=0\},
$$
and
\begin{equation}
    \gamma_{a, b, c}: = \limsup_{|S|\to +\infty}\frac{ T_{a,b,c}(S)} {|S|^{2}}.
\end{equation}
Perhaps it is surprising that the limit $\gamma_{a, b, c}$ is known only for a few $(a,b,c)$.
The first known example is $(a, b, c) = (1,1,-1) $ and it was proved by Hardy and Littlewood \cite{hardy1928notes, hardy1952inequalities} and Gabrie\cite{gabriel1931rearrangement} that $\gamma_{1,1,-1} = \frac{3}{4}$ when $|S|$ is odd, and by Lev \cite{lev1998number} when $|S|$ is even. The second known example is $(1,1,-2)$ which is the same as counting 3-term arithmetic progressions in $S$, and Green and Sisask in \cite{green2008maximal} proved that $\gamma_{1,1,-2} = \frac{1}{2}$. And later it was shown that $\gamma_{1, -1, -2}=1/2$ by \cite{lev2014solving} Lev and Pinchasi. It is easy to see that one can choose a centered interval to obtain the maximal number of solutions in the above cases, i.e. 
\[S= \{-\Big\lfloor \frac{N}{2}\Big\rfloor, \dots,  \Big\lfloor \frac{N}{2} \Big\rfloor \}.\]
Aaronson~\cite{Aa} showed that $\gamma_{a,b,c} \geq 1/12$ for nonzero $a,b,c$ and that $1/12$ is the best possible
universal constant. Nevertheless, determining $\gamma_{a,b,c}$ for a
specific fixed triple remains difficult even for very small coefficients.

The first goal of our paper is to make some progress in determining the next non-obvious case $\gamma_{1,1,-3}$. 

\begin{theorem}\label{THM: 113}
    Let $\gamma_{1,1,-3}$ be defined as above. Then there exists some $\delta>0$ such that 
    \[\frac{5}{13} \le \gamma_{1,1,-3} \le \frac{1}{2}-\delta.\]
\end{theorem}
 Note that the lower bound is strictly bigger than $1/3$, which is what the interval example would give. This shows that the centered interval is no longer the optimal choice. Our upper bound is certainly not optimal, but it breaks the natural barrier of $1/2$. We conjecture that the lower bound is sharp.  
\begin{conjecture}\label{conj}
    Let $\gamma_{1,1,-3}$ be defined as above. Then 
    \[\gamma_{1,1,-3} = \frac{5}{13}. \]
\end{conjecture}
The main result of our paper is the following structural result, Theorem \ref{thm:main}. Roughly speaking, it gives finite-dimensional description of $\gamma_{a_1,a_2,a_3}$ and therefore determines the structure of optimal sets. 
\begin{definition}\label{defn:G}
For $\alpha,\beta,\gamma \geq 0$, define $G(\alpha,\beta,\gamma)$ to be the are of the region
$$
\{(x,y) \in \R^2: |x| \leq \alpha/2, |y| \leq \beta/2, |x+y| \leq \gamma/2\}.
$$
\end{definition}

It can be explicitly computed: $G$ is symmetric in all three variables, and in the case $\alpha \leq \beta \leq \gamma$ we have
$$
G(\alpha,\beta,\gamma) = \begin{cases} \alpha\beta & \text{if }\alpha+\beta \leq \gamma \\ \alpha\beta - \frac{1}{4}(\alpha+\beta-\gamma)^2 & \text{if }\alpha+\beta\geq\gamma. \end{cases}
$$

\begin{theorem}\label{thm:main}
Let $a_1,a_2,a_3$ be fixed nonzero integers with $\gcd(a_1,a_2,a_3) = 1$. For a positive integer $q$ and a function $g: \Z/q\Z \rightarrow [0,1]$ with $\sum_r g(r) = 1$, define
$$
\gamma_{a_1,a_2,a_3}(g) = \frac{1}{|a_1a_2a_3|} \sum_{\substack{r_1,r_2,r_3\Mod{q} \\ a_1r_1+a_2r_2+a_3r_3\equiv 0\Mod{q}}} G(|a_1|g(r_1), |a_2|g(r_2), |a_3|g(r_3)).
$$
Define $\gamma_{a_1,a_2,a_3}(q)$ to be the maximum of $\gamma_{a_1,a_2,a_3}(g)$ over all such functions $g$. Then
$$
\gamma_{a_1,a_2,a_3} = \sup_q \gamma_{a_1,a_2,a_3}(q).
$$
Moreover, in the supremum over $q$ above we may restrict to those $q$ all of whose prime factors divide $a_1a_2a_3$.
\end{theorem}

The quantity $\gamma_{a_1,a_2,a_3}(g)$ in the statement appears naturally when considering $T_{a_1,a_2,a_3}(A)$ when $A$ is a balanced interval on each residue class modulo $q$:

\begin{lemma}\label{lem:local-sets}
Let $a_1,a_2,a_3$ be fixed nonzero integers with $\gcd(a_1,a_2,a_3) = 1$. Let $q$ be a positive integer and let $g: \Z/q\Z \rightarrow [0,1]$ be a function with $\sum_r g(r) = 1$. Then for $N$ sufficiently large, the set $A \subset \Z$ defined by
$$
A = \bigcup_{r=1}^q \Big\{n \equiv r\Mod{q}: \Big|\frac{n}{N}\Big| \leq \frac{g(r)}{2}\Big\}
$$
satisfies
$$
\frac{T_{a_1,a_2,a_3}(A)}{|A|^2} \sim \gamma_{a_1,a_2,a_3}(g).
$$
\end{lemma}

In view of the above lemma, the lower bound $\gamma_{a_1,a_2,a_3} \geq \gamma_{a_1,a_2,a_3}(q)$ holds for every $q$. Theorem \ref{thm:main} says, roughly speaking, that asymptotically extremal values of $\gamma_{a_1,a_2,a_3}$ are achieved  by taking, in each residue class modulo $q$, an interval
centered at the origin whose length is prescribed by $g$.
 
As an application of Theorem \ref{thm:main}, we prove the following:

\begin{theorem}\label{thm:1/5}
Let $a$ be a positive integer. Then
$$
\gamma_{1,1,a} = \frac{1}{5} + O\Big(\frac{\log\log a}{\log a}\Big).
$$
\end{theorem}

For studying $\gamma_{1,1,3}$ or $\gamma_{1,-1,3}$, Theorem \ref{thm:main} reduces to an optimization problem over $\Z/q\Z$, where $q$ is a power of $3$. For $q=3$, the solutions to the corresponding optimization problems are given as follows.

\begin{theorem}\label{thm:gamma113}
We have
$$
\gamma_{1,-1,\pm 3}(3) = \frac{47}{122}, \ \ \gamma_{1,1,\pm 3}(3) = \frac{5}{13}.
$$
\end{theorem}

We conjecture that the modulo $3$ construction is already globally optimal; i.e. $\gamma_{1,\pm 1, \pm 3} = \gamma_{1,\pm 1,\pm 3}(3)$ (leading to Conjecture \ref{conj}).

After the completion of this work, Korsky~\cite{Korsky} obtained the
strikingly similar upper bound
$$
    \gamma_{1,2,-3}\leq\frac{47}{122}
$$
for the translation-invariant equation $x+2y=3z$.
It is perhaps not surprising that the constant $47/122$ from Korky's work and from Theorem \ref{thm:gamma113} arise from the same quadratic optimization problem. It would be interesting to understand this connection more systematically, and in
particular whether the recursive ideas in~\cite{Korsky} can be adapted
to resolve the value of $\gamma_{1,-1,\pm 3}$.

\subsection*{Organization of the paper}

In Section \ref{sec:113} we prove the nontrivial upper bound for $\gamma_{1,1,-3}$ in Theorem \ref{THM: 113}.  In Section \ref{sec:local} we prove Lemma \ref{lem:local-sets} connecting $\gamma_{a_1,a_2,a_3}(g)$ with $T_{a_1,a_2,a_3}(A)$. In Section \ref{sec:1/5} we prove the asymptotic behavior of $\gamma_{1,1,a}$ as $a\rightarrow\infty$ in Theorem \ref{thm:1/5}. Sections \ref{sec:pf-outline}--\ref{sec:proof3} contains the proof of Theorem \ref{thm:main}, with an outline in Section \ref{sec:pf-outline}. Its proof uses arithmetic regularity to reduce an arbitrary set to a structured model,  and then studies the resulting optimization problem over $\Z/q\Z \times \R \times (\R/\Z)^d$.  Finally, in Section \ref{sec:47/122} we
study the finite optimization problems modulo $3$ in Theorem~\ref{thm:gamma113}.

\section{Proof of  Theorem~\ref{THM: 113}}\label{sec:113}

The lower bound $\gamma_{1,1,-3} \geq 5/13$ follows from the observation $\gamma_{1,1,-3} \geq \gamma_{1,1,-3}(3)$ (see the discussion following Lemma \ref{lem:local-sets}) combined with Theorem \ref{thm:gamma113}.

For the upper bound, we first give a short proof that $\gamma_{1,1,-3} \leq \frac{1}{2}$. For sets $A, B, C$, let $T(A, B, C)$ denote the number of solutions to $x + y = z$ with $x \in A, y \in B, z \in C$. We only need the trivial bound
$$T(A, B, C) \leq \min(|A||B|, |B||C|, |C||A|).$$
\begin{lemma}
    \label{lem:half-bound}
  For any finite $S\subset \mathbb{Z}$ we have
    \begin{equation}\label{1/2}
  T(S, S, 3 \cdot S) \leq \frac{1}{2} |S|^2 + 1.      
    \end{equation}  
\end{lemma}
\begin{proof}
    We argue by induction on $|S|$. The base case $|S| = 1$ is trivial. For the induction step, let $S_i$ denote the elements of $S$ congruent to $i$ modulo $3$; Without loss of generality, we can assume $S_0 \neq S$. Then we have
\begin{equation}\label{eqn: iden}
   T(S, S, 3 \cdot S) = 2 T(S_1, S_2, 3 \cdot S) + T(S_0, S_0, 3 \cdot (S_1 \cup S_2)) + T(S_0, S_0, 3 \cdot S_0). 
\end{equation}    
    Using the trivial bound on the first two terms, and the induction hypothesis on the third, we obtain
    $$T(S, S, 3 \cdot S) \leq 2 |S_1| |S_2| + |S_0|(|S_1| + |S_2|) + \frac{1}{2} |S_0|^2 + 1.$$
    By the AM-GM inequality, we conclude that 
    $$T(S, S, 3 \cdot S) \leq \frac{1}{2} (|S_0| + |S_1| + |S_2|)^2 + 1$$
    as desired.
\end{proof}
We now show how an improvement on \eqref{1/2} is possible. This needs the following non-trivial bound (\cite[Lemma 2]{lev2014solving}).
\begin{lemma}
    If $|A| + |C| \geq |B|$ and $|B| + |C| \geq |A|$, then
    \label{lemma ABC}
    $$T(A, B, C) \leq |A||B| - \frac{1}{4} \max(|A| + |B| - |C|, 0)^2 + 1.$$
\end{lemma}
\begin{proof}
    Note that it is enough to consider the case $|C|\le |A|+|B|$. Again the proof is identical to that of (\cite[Lemma 2]{lev2014solving}) using induction on $|A|+|B|-|C|$ and noting that the case where either $A$ or $B$ is empty is trivial.
\end{proof}
First, the trivial bound on the second term in \eqref{eqn: iden} can be improved to
$$T(S_0, S_0, 3 \cdot (S_1 \cup S_2)) \leq |S_0| |S_1 \cup S_2| - \frac{1}{4} |S_1 \cup S_2|^2+1,$$
unless $|S_1 \cup S_2| \geq 2 |S_0|$. In that case, we have an even better improvement
$$T(S_0, S_0, 3 \cdot (S_1 \cup S_2)) \leq |S_0|^2.$$
Thus we already get a win unless either $S_0$ is very small, or $S_1 \cup S_2$ is very small. 

We first deal with the case when $S_0$ is very small. In the most extreme case, $S_0 = \emptyset$, so we want to win over the trivial bound
$$T(S_1, S_2, 3 \cdot (S_1 \cup S_2)) \leq |S_1||S_2| \leq \frac{1}{4} \left(|S_1| + |S_2|\right)^2.$$
We next state two lemmas and defer their proofs to Appendix. 
\begin{lemma}
\label{lemma ABAcupB}
For any set $A$ consisting of integers congruent to $1$ modulo $3$, and $B$ consisting of integers congruent to $2$ modulo $3$, we have 
$$T(A, B, 3 \cdot (A \cup B)) \leq \frac{6}{25} \left(|A| + |B|\right)^2 + 2.$$
\end{lemma}

We say $S$ is \textit{winnable} if $|S_1| + |S_2| \geq \frac{1}{3}|S|$.
\begin{lemma}\label{lem 3.4}
    If $S$ is winnable, then for some absolute constant $\varepsilon > 0$, we have
    $$T(S, S, 3 \cdot S) \leq \left(\frac{1}{2} - \varepsilon\right) |S|^2 + 6.$$
\end{lemma}

It remains to consider the case when $S$ is not winnable, where we need a final observation.
\begin{lemma}
If $T(S_0, S_0, 3 \cdot (S_1 \cup S_2)) \geq \frac{2}{3} |S_0| |S_1 \cup S_2|$, then $S' = \frac{1}{3} \cdot S_0$ is winnable.   
\end{lemma}
\begin{proof}
    As
    $$T(S_0, S_0, 3 \cdot (S_1 \cup S_2)) = T\left(S', S', S_1 \cup S_2\right)$$
    there exists some $a \in  S_1 \cup S_2$ such that for at least $\frac{2}{3} |S'|$ of the $b \in S'$, $a - b$ also lies in $S'$. For any such pair $(b, a - b)$, at least one of $a$ and $b - a$ lies in $S_1' \cup S_2'$.
\end{proof}
We can finally show that
\begin{proposition}
    For some absolute constant $\delta > 0$, we have
    $$T(S, S, 3 \cdot S) \leq \left(\frac{1}{2} - \delta\right) |S|^2 + 3.$$
\end{proposition}
\begin{proof} Let $\varepsilon = \frac{49}{2601}$ as in the previous lemma and let $\delta=\frac{4\varepsilon}{9}$. We argue by induction on $|S|$. Assume that the result holds for all subsets of $S$.

    If $S$ is winnable, then we are already done. Otherwise, $S$ is not winnable, which means $|S_1| + |S_2| \leq \frac{1}{3} |S|$. We may assume that $S_0\neq S$. Recall that
    $$T(S, S, 3 \cdot S) = 2 T(S_1, S_2, 3 \cdot S) + T(S_0, S_0, 3 \cdot (S_1 \cup S_2)) + T(S_0, S_0, 3 \cdot S_0).$$
    Bounding the first term trivially and the third term by the induction hypothesis, we get
    \begin{equation}\label{eqn: *}
       T(S, S, 3 \cdot S) \le \frac{1}{2} (|S_1| + |S_2|)^2 + T(S_0, S_0, 3 \cdot (S_1 \cup S_2)) + \left(\frac{1}{2} - \delta\right) |S_0|^2 + 3. 
    \end{equation}
    If $T(S_0, S_0, 3 \cdot (S_1 \cup S_2)) \leq \frac{2}{3} (|S_1| + |S_2|) |S_0|$, then we are done, since
    $$\frac{1}{2} (|S_1| + |S_2|)^2 + T(S_0, S_0, 3 \cdot (S_1 \cup S_2)) \leq \frac{1}{2} (|S_1| + |S_2|)^2 + \frac{2}{3} (|S_1| + |S_2|) |S_0|$$
    and we have
    $$\frac{1}{2} (|S_1| + |S_2|)^2 + T(S_0, S_0, 3 \cdot (S_1 \cup S_2)) \leq (\frac{1}{2} - \delta)((|S_1| + |S_2|)^2 + (1-2\delta)(|S_1| + |S_2|) |S_0|)$$
    and substituting into \eqref{eqn: *} gives what we want.
    
    Otherwise, by the preceding lemma, $\frac{1}{3} \cdot S_0$ is winnable, which allows us to upgrade \eqref{eqn: *} into
    $$T(S, S, 3 \cdot S) \le \frac{1}{2} (|S_1| + |S_2|)^2 + T(S_0, S_0, 3 \cdot (S_1 \cup S_2)) + \left(\frac{1}{2} - \varepsilon\right) |S_0|^2 + 3.$$
    Now we can also bound the second term trivially
    $$T(S, S, 3 \cdot S) \le \frac{1}{2} (|S_1| + |S_2|)^2 + |S_0|(|S_1| + |S_2|) + \left(\frac{1}{2} - \varepsilon\right) |S_0|^2 + 3.$$
    By our assumption that $|S_1| + |S_2| \leq \frac{1}{3} |S|$, we conclude that
    $$T(S, S, 3 \cdot S) \le \left(\frac{1}{2} - \delta\right) |S|^2 + 3$$
    as desired.
\end{proof}

\section{The construction in Lemma \ref{lem:local-sets}}\label{sec:local}

\begin{proof}[Proof of Lemma \ref{lem:local-sets}]
We prove a natural asymmetric extension of Lemma \ref{lem:local-sets}. Let $g_1,g_2,g_3: \Z/q\Z \rightarrow [0,1]$ be three functions. For large $N$ and $i \in \{1,2,3\}$, define
$$
A^{(i)} = \bigcup_{r=1}^q A_r^{(i)}, \ \ A_r^{(i)} := \Big\{n \equiv r\Mod{q}: \Big|\frac{n}{N}\Big| \leq \frac{g_i(r)}{2}\Big\}.
$$
Then
$$
|A^{(i)}| \sim  \frac{N}{q} \sum_r g_i(r).
$$
We will prove that
\begin{equation}\label{eq:local-sets-1}
T_{a_1,a_2,a_3}(A^{(1)}, A^{(2)}, A^{(3)}) \sim \Big(\frac{N}{q}\Big)^2 \gamma_{a_1,a_2,a_3}(g_1,g_2,g_3),
\end{equation}
where $T_{a_1,a_2,a_3}(A^{(1)}, A^{(2)}, A^{(3)})$ denotes the number of solutions to $a_1x_1 + a_2x_2 + a_3x_3 = 0$ with $x_i \in A^{(i)}$, and 
$$
\gamma_{a_1,a_2,a_3}(g_1,g_2,g_3) := \frac{1}{|a_1a_2a_3|} \sum_{\substack{r_1,r_2,r_3\Mod{q} \\ a_1r_1+a_2r_2+a_3r_3\equiv 0\Mod{q}}} G(|a_1|g_1(r_1), |a_2|g_2(r_2), |a_3|g_3(r_3)).
$$
Since
$$
T_{a_1,a_2,a_3}(A^{(1)}, A^{(2)}, A^{(3)}) = \sum_{\substack{r_1,r_2,r_3\Mod{q} \\ a_1r_1+a_2r_2+a_3r_3\equiv 0\Mod{q}}} T_{a_1,a_2,a_3}(A^{(1)}_{r_1}, A^{(2)}_{r_2}, A^{(3)}_{r_3}),
$$
it suffices to show that
\begin{equation}\label{eq:local-sets-2}
 T_{a_1,a_2,a_3}(A^{(1)}_{r_1}, A^{(2)}_{r_2}, A^{(3)}_{r_3}) \sim \frac{N^2}{q^2|a_1a_2a_3|}  G(|a_1|g_1(r_1), |a_2|g_2(r_2), |a_3|g_3(r_3))
\end{equation}
for all $r_1,r_2,r_3\Mod{q}$ with $a_1r_1+a_2r_2+a_3r_3 \equiv 0\Mod{q}$. Fix such $r_1,r_2,r_3\Mod{q}$ for the rest of the proof.

Choose $r_1^*,r_2^*,r_3^*\in\Z$ with $r_i^*\equiv r_i\Mod{q}$ for each $i$, such that $a_1r_1^*+a_2r_2^*+a_3r_3^*=0$. This can be achieved as follows. Start with any $r_1',r_2',r_3'\in\Z$ with $r_i'\equiv r_i\Mod{q}$. Then $a_1r_1'+a_2r_2'+a_3r_3'\equiv 0\Mod{q}$. Since $\gcd(a_1,a_2,a_3)=1$, one can find $m_1,m_2,m_3 \in \Z$ such that 
$$
a_1m_1+a_2m_2+a_3m_3 = -\frac{a_1r_1'+a_2r_2'+a_3r_3'}{q}.
$$
Then set $r_i^* = r_i' + qm_i$.

If $n_1,n_2,n_3$ with $n_i \in A^{(i)}_{r_i}$ satisfy $a_1n_1+a_2n_2+a_3n_3 = 0$, then $x_i = a_in_i$ satisfy
$$
x_1 + x_2 + x_3 = 0, \ \ |x_i| \leq \frac{|a_i|}{2}g_i(r_i)N, \ \ x_i \equiv a_ir_i\Mod{a_iq}. 
$$
Moreover, the congruence condition $x_3 \equiv a_3r_3\Mod{a_3q}$ is equivalent to $x_3 \equiv a_3r_3^*\Mod{a_3q}$, which is in turn equivalent to
$$
x_1+x_2 \equiv a_1r_1^*+a_2r_2^*\Mod{a_3q}.
$$
Hence the number of such $(n_1,n_2,n_3)$ is the size of the intersection of the convex body $N \cdot \Omega$, where
$$
\Omega = \Big\{(x_1, x_2) \in \R^2: |x_i| \leq \frac{|a_i|}{2}g_i(r_i), \ \ |x_1+x_2|\leq \frac{|a_3|}{2}g_3(r_3)\Big\},
$$
and the translate $(a_1r_1^*, a_2r_2^*) + \Gamma$ of the lattice
$$
\Gamma = \{(x_1, x_2) \in \Z^2: x_i \equiv 0\Mod{a_iq}, \ \ x_1+x_2\equiv 0\Mod{a_3q}\}.
$$
It follows that
$$
T_{a_1,a_2,a_3}(A^{(1)}_{r_1},A^{(2)}_{r_2},A^{(3)}_{r_3}) = |(N\cdot \Omega) \cap ((a_1r_1^*,a_2r_2^*) + \Gamma)| \sim \frac{\vol(\Omega)}{\vol(\Gamma)} N^2.
$$
By Definition \ref{defn:G}, we have 
$$
\vol(\Omega) = G(|a_1|g_1(r_1), |a_2|g_2(r_2), |a_3|g_3(r_3)). 
$$
If $(x_1,x_2) \in \Gamma$, then from $x_2 \equiv 0\Mod{a_2q}$ and $x_2 \equiv -x_1\Mod{a_3q}$ it follows that $x_1 \equiv 0\Mod{q\cdot \gcd(a_2,a_3)}$. Since $a_1$ is coprime with $\gcd(a_2,a_3)$, we have $x_1 \equiv 0\Mod{qa_1\cdot \gcd(a_2,a_3)}$. Once such an $x_1$ is chosen, $x_2$ must lie in a specific residue class modulo $q\cdot \lcm(a_2,a_3)$. It follows that 
$$
\vol(\Gamma) = q^2|a_1 \cdot \gcd(a_2,a_3) \cdot \lcm(a_2,a_3)| = q^2 |a_1a_2a_3|.
$$
This leads to the desired estimate \eqref{eq:local-sets-2}.
\end{proof}

Using this construction, one can deduce the following lemma which will be used later.

\begin{lemma}\label{lem:G-sumbound}
For any functions $f_1,f_2,f_3: \Z/q\Z \rightarrow \R_{\geq 0}$ we have
$$
\frac{1}{|a_1a_2a_3|} \sum_{\substack{r_1,r_2,r_3\Mod{q} \\ a_1r_1+a_2r_2+a_3r_3\equiv 0\Mod{q}}} G(|a_1|f_1(r_1), |a_2|f_2(r_2), |a_3|f_3(r_3)) \leq G(F_1, F_2, F_3),
$$
where $F_i = \sum_s f_i(s)$.
\end{lemma}

\begin{proof}
Note that the left-hand side is precisely $\gamma_{a_1,a_2,a_3}(f_1,f_2,f_3)$ above. Let $N$ be sufficiently large and define $A^{(i)}$ for $i \in \{1,2,3\}$ as in the proof of Lemma \ref{lem:local-sets} above. By \eqref{eq:local-sets-1} we have
$$
\gamma_{a_1,a_2,a_3}(f_1,f_2,f_3) \sim \Big(\frac{q}{N}\Big)^2 T_{a_1,a_2,a_3}(A^{(1)}, A^{(2)}, A^{(3)}).
$$
Letting $B^{(i)} = a_i \cdot A^{(i)} := \{a_in : n \in A^{(i)}\}$, we have by definition
$$
 T_{a_1,a_2,a_3}(A^{(1)}, A^{(2)}, A^{(3)}) =  T_{1,1,1}(B^{(1)}, B^{(2)}, B^{(3)}).
$$
By the rearrangement inequality \cite[Theorem 1]{lev1998number}, $T_{1,1,1}(B^{(1)},B^{(2)},B^{(3)})$ is maximized when each $B^{(i)}$ is a balanced interval around $0$ of length $|B^{(i)}| = |A^{(i)}| \sim (N/q) F_i$, and hence
$$
T_{1,1,1}(B^{(1)}, B^{(2)}, B^{(3)}) \leq (1+o(1)) G(|B^{(1)}|, |B^{(2)}|, |B^{(3)}|) \leq (1+o(1)) \Big(\frac{N}{q}\Big)^2 G(F_1, F_2, F_3).
$$
The conclusion follows.
\end{proof}

\section{Proof of Theorem~\ref{thm:1/5}}\label{sec:1/5}

Let $\eps > 0$ be a small parameter. By Theorem \ref{thm:main}, we have $\gamma_{1,1,a} = \sup_g \gamma(g)$, where the supremum is taken over all functions $g: \Z/q\Z \rightarrow [0,1]$ with $\sum_x g(x) = 1$ for some $q = a^k$, and
$$
\gamma(g) = \gamma_{1,1,a}(g) = \frac{1}{a} \sum_{\substack{x,y,z\Mod{q}\\ x+y+az\equiv 0\Mod{q}}} G(g(x), g(y), ag(z)).
$$

First we show that we can restrict to those terms with $g(x), g(y), ag(z)$ within a multiplicative factor of $\eps$ of each other in the sum above at the cost of an $O(\eps)$ error. 
Indeed, the total contribution to $\gamma(g)$ from those terms with $g(x) < \eps g(y)$ is at most
$$
\frac{1}{a} \sum_{\substack{x,y,z\Mod{q}\\ x+y+az\equiv 0\Mod{q}}} \eps g(y)\cdot ag(z) = \eps \sum_{y,z} g(y) g(z) = \eps,
$$
where we used the bound $G(g(x), g(y), ag(z)) \leq g(x) \cdot ag(z)$. Similarly, the total contribution from those terms with $g(y) < \eps g(x)$ is also at most $\eps$. Next, the total contribution from those terms with $g(x) < \eps ag(z)$ is at most
$$
\frac{1}{a} \sum_{\substack{x,y,z\Mod{q}\\ x+y+az\equiv 0\Mod{q}}} \eps ag(z) g(y) = \eps \sum_{y,z} g(y)g(z) = \eps,
$$
where we used the bound $G(g(x), g(y), ag(z)) \leq g(x)g(y)$. Finally, the total contribution from those terms with $ag(z) < \eps g(x)$ is at most
$$
\frac{1}{a} \sum_{\substack{x,y,z\Mod{q}\\ x+y+az\equiv 0\Mod{q}}} \eps g(x) g(y) = \eps \sum_{\substack{x,y\Mod{q} \\ a \mid x+y}} g(x) g(y) \leq \eps,
$$
where we used the bound $G(g(x), g(y), ag(z)) \leq g(y) \cdot ag(z)$ and the fact that there are $a$ choices of $z\Mod{q}$ satisfying $x+y+az\equiv 0\Mod{q}$ once $x,y$ are fixed with $a\mid x+y$. Similarly, the total contribution from those terms with $g(y) < \eps ag(z)$ or $ag(z) < \eps g(y)$ is also at most $\eps$. Hence we have
$$
\gamma(g) = \frac{1}{a} \sum_{\substack{x,y,z\Mod{q}\\ x+y+az\equiv 0\Mod{q} \\ \min(g(x), g(y), ag(z)) \geq \eps \max(g(x),g(y),ag(z))}} G(g(x), g(y), ag(z)) + O(\eps)
$$

Let $\lambda \in [1,a]$ be a parameter to be chosen later. For $i \geq 0$, let 
$$
X_i = \{x \in \Z/q\Z: g(x) \in (\lambda a^{-i-1}, \lambda a^{-i}]\}, \ \ X_i' = \{x \in \Z/q\Z: g(x) \in (\eps^{-1}\lambda a^{-i-1}, \eps \lambda a^{-i}]\}.
$$
Then $X_0,X_1,\cdots$ forms a partition of $\Z/q\Z$, and $X_i' \subset X_i$ for each $i$. Let $E = E_{\lambda}:= \cup_i (X_i\setminus X_i')$ be the set of exceptional $x$. The contribution to $\gamma(g)$ from those terms with at least one of $x,y,z$ lying in $E$ is at most
$$
\frac{1}{a}\sum_{\substack{x,y,z\Mod{q} \\ x+y+az\equiv 0\Mod{q} \\ x \in E}} g(x) \cdot ag(z) + \frac{1}{a}\sum_{\substack{x,y,z\Mod{q} \\ x+y+az\equiv 0\Mod{q} \\ y \in E}} g(y) \cdot ag(z) + \frac{1}{a}\sum_{\substack{x,y,z\Mod{q} \\ x+y+az\equiv 0\Mod{q} \\ z \in E}} g(x) \cdot ag(z),
$$
and each of the three terms above is at most $g(E) := \sum_{x \in E}g(x)$.

Now let's choose $\lambda \in [1,a]$ so that $g(E)$ is small. If $x \in X_i\setminus X_i'$ then
$$
\eps g(x)a^{i+1} \leq \lambda < g(x)a^{i+1}\text{ or }g(x)a^i \leq \lambda < \eps^{-1} g(x)a^i,
$$
and hence
$$
\int_1^a g(E_\lambda) \frac{d\lambda}{\lambda} = \sum_x g(x) \sum_i \Big(\int_{\eps g(x)a^{i+1}}^{g(x)a^{i+1}} 1_{\lambda \in [1,a]}\frac{d\lambda}{\lambda} + \int_{g(x)a^i}^{\eps^{-1}g(x)a^i} 1_{\lambda \in [1,a]}\frac{d\lambda}{\lambda}\Big).
$$
Each of the integrals on the right-hand side above is at most $\log\eps^{-1}$, and, since $\lambda \in [1,a]$, for each $x$ there are at most two values of $i$ for which the integral does not vanish. Hence the expression above is at most
$$
4\log \eps^{-1} \sum_x g(x) = 4\log \eps^{-1}.
$$
Hence there exists $\lambda$ such that $E = E_\lambda$ satisfies
$$
g(E) \leq \frac{4\log \eps^{-1}}{\log a}.
$$
We make this choice of $\lambda$ with $\eps = 1/(\log a)$, which leads to
$$
\gamma(g) = \frac{1}{a} \sum_{\substack{x,y,z\Mod{q}\\ x+y+az\equiv 0\Mod{q} \\ \min(g(x), g(y), ag(z)) \geq \eps \max(g(x),g(y),ag(z)) \\ x,y,z\notin E}} G(g(x), g(y), ag(z)) + O\Big(\frac{\log\log a}{\log a}\Big).
$$

If $x,y,z \notin E$ satisfies $\min(g(x), g(y), ag(z)) \geq \eps \max(g(x),g(y),ag(z))$, then $x \in X_i'$ for some $i$, and hence $g(x) \in (\eps^{-1}\lambda a^{-i-1}, \eps\lambda a^{-i}]$. Then $g(y), ag(z) \in (\lambda a^{-i-1}, \lambda a^{-i}]$, so that $y \in X_i$ and $z \in X_{i+1}$. Hence $y \in X_i'$ and $z \in X_{i+1}'$ since $y,z \notin E$. It follows that
$$
\gamma(g) \leq \frac{1}{a} \sum_i \sum_{\substack{x,y \in X_i', z \in X_{i+1}'\\ x+y+az\equiv 0\Mod{q}}} G(g(x), g(y), ag(z)) + O\Big(\frac{\log\log a}{\log a}\Big).
$$
For each $i$, by Lemma \ref{lem:G-sumbound} (applied to the functions $g(x)1_{x \in X_i'}, g(y)1_{y \in X_i'}, g(z)1_{z\in X_{i+1}'}$) we have
$$
\frac{1}{a} \sum_{\substack{x,y \in X_i', z \in X_{i+1}'\\ x+y+az\equiv 0\Mod{q}}} G(g(x), g(y), ag(z))  \leq G(g(X_i'), g(X_i'), g(X_{i+1}')).
$$
It follows that
$$
\gamma(g) \leq \sum_{i \geq 0} G(g(X_i'), g(X_i'), g(X_{i+1}')) + O\Big(\frac{\log\log a}{\log a}\Big) \leq \frac{1}{5} + O\Big(\frac{\log\log a}{\log a}\Big),
$$
where the second inequality follows from the lemma below applied to the sequence $\{g(X_i')\}_{i\geq 0}$. This completes the proof.

\begin{lemma}
    \label{lem:optimize}
    Suppose $n_1, n_2, \cdots, n_k$ is a finite sequence of non-negative real numbers, with $ n_1 + \cdots + n_k \leq 1$. Then we have
    $$\sum_{i = 1}^{k - 1} G(n_i, n_i,n_{i + 1}) \leq \frac{1}{5}.$$
\end{lemma}
\begin{proof}
For convenience, we write $f(x,y) = G(x,x,y)$ so that
$$
f(x,y) = \begin{cases} x^2 & \text{if }y \geq 2x, \\ xy-\frac{1}{4}y^2 & \text{if }y \leq 2x.\end{cases}
$$
Assume the contrary. Let $k$ be minimal such that there are $m_1, \cdots, m_k$ satisfying $m_1+\cdots+m_k \leq 1$ and
\[\sum_{i=1}^{k-1}f(m_i,m_{i+1})>\frac{1}{5}
\] For that particular value of $k$ take $m_1,m_2,\ldots ,m_k$ to be such that $\sum_{i=1}^{k-1}f(m_i,m_{i+1})$ is maximal. We claim that $k\le 2$. First, note that all of the $m_i$ are non-zero. Otherwise, we can increase the sum by removing zeros.

Assume that $k\ge 3$. We will consider various sequences $m_i'$ to derive a contradiction. First, we consider 
$$m_i' = \begin{cases}
m_i, &i \leq k - 3, \\
m_{k - 2} + m_k, &i = k - 2, \\ 
m_{k - 1}, & i = k - 1. \\
\end{cases}$$
Then we have
\begin{align*}
\sum_{i = 1}^{k - 2} f(m_i', m_{i + 1}') - \sum_{i = 1}^{k - 1} f(m_i, m_{i + 1}) 
=& (f(m_{k - 3}, m_{k - 2}') - f(m_{k - 3}, m_{k - 2})) \\
+& (f(m_{k - 2} + m_k, m_{k - 1}) - f(m_{k - 2}, m_{k - 1}) - f(m_{k - 1}, m_{k}))    .
\end{align*}
By minimality of $k$ the left hand side must be negative. On the right hand side, the first summand is non-negative since $f$ is non-decreasing in the second variable (note that if $k=3$ the first summand is absent). Therefore, we must have
$$f(m_{k - 2} + m_k, m_{k - 1}) - f(m_{k - 2}, m_{k - 1}) \leq f(m_{k - 1}, m_{k}).$$
This implies $m_{k - 1} > 2 m_{k - 2}$, for otherwise we have
$$f(m_{k - 2} + m_k, m_{k - 1}) - f(m_{k - 2}, m_{k - 1}) = m_k m_{k - 1}$$
while
$$f(m_{k - 1}, m_{k}) < m_{k - 1} m_{k}$$
which is a contradiction.

Next, we consider
$$m_i' = \begin{cases}
m_i, &i \leq k - 3, \\
m_{k - 2} + m_{k - 1}, &i = k - 2, \\ 
m_{k}, &i = k - 1.
\end{cases}$$
Then we have
\begin{align*}
\sum_{i = 1}^{k - 2} f(m_i', m_{i + 1}') - \sum_{i = 1}^{k - 1} f(m_i, m_{i + 1}) 
=& (f(m_{k - 3}, m_{k - 2}') - f(m_{k - 3}, m_{k - 2})) \\
+& (f(m_{k - 2} + m_{k - 1}, m_{k}) - f(m_{k - 2}, m_{k - 1}) - f(m_{k - 1}, m_{k}))    .
\end{align*}
Again, the first summand is non-negative, so we must have
$$f(m_{k - 2} + m_{k - 1}, m_{k}) - f(m_{k - 1}, m_{k}) \leq  f(m_{k - 2}, m_{k - 1})  = m_{k - 2}^2.$$
The slope of $f(m,n)$ in the first variable is equal to $\min (2m,n)$ so we have
$$f(m_{k - 2} + m_{k - 1}, m_{k}) - f(m_{k - 1}, m_{k}) \geq m_{k - 2} \min(2m_{k - 1}, m_k).$$
Therefore, we have $m_{k - 2}^2 \geq m_{k - 2} \min(m_{k - 1}, m_k)$, which implies that $m_{k - 2} \geq \min(m_{k - 1}, m_k)$. However, we have $m_{k - 1} > 2m_{k - 2}$, so we must have $m_{k - 2} \geq m_k$. Thus we get
$$m_{k - 1} > 2m_{k - 2} \geq 2 m_k.$$
Finally, we take $\varepsilon > 0$ sufficiently small, and let
$$m_i' = \begin{cases}
m_i, &i \leq k - 3, \text{ or } i=k\\
m_{k - 2} + \varepsilon, &i = k - 2, \\ 
m_{k - 1} - \varepsilon, &i = k - 1, \\
\end{cases}.$$
Then we have
\begin{align*}
&\sum_{i = 1}^{k - 1} f(m_i', m_{i + 1}') - \sum_{i = 1}^{k - 1} f(m_i, m_{i + 1}) \\
=& (f(m_{k - 3}, m_{k - 2}') - f(m_{k - 3}, m_{k - 2})) 
+ (f(m_{k - 2}', m_{k - 1}') - f(m_{k - 2}, m_{k - 1}) + f(m_{k - 1}', m_{k}) -  f(m_{k - 1}, m_{k})).  
\end{align*}
On the right hand side, the first term is again non-negative. For the second term, we can compute that
$$f(m_{k - 2}', m_{k - 1}') - f(m_{k - 2}, m_{k - 1}) = (m_{k - 2}')^2 - m_{k - 2}^2 \geq 2\varepsilon m_{k - 2}$$
while
$$f(m_{k - 1}', m_{k}) -  f(m_{k - 1}, m_{k}) = (m_{k - 1}' - m_{k - 1}) m_k = -\varepsilon m_k.$$
So their sum is at least $\varepsilon (2 m_{k - 2} - m_k) > 0$, contradiction.

We conclude that $k\le 2$. Let $b=m_2$. Then we obtain
$$\sum_{i = 1}^{k - 1} f(n_i, n_{i + 1}) \leq f(1 - b,b ).$$
As 
$$f(1 - b, b) = \begin{cases}
    b - \frac{5b^2}{4}, &b \leq \frac{2}{3} \\
    (1 - b)^2, &b \geq \frac{2}{3}
\end{cases}$$
it is clear that the maximum value of $f(1 - b, b)$ is achieved at $b=\frac{2}{5}$ and is equal to $\frac{1}{5}$. This gives us a contradiction as desired.
\end{proof}

\begin{remark}
Aaronson's work \cite{Aa} shows that $\gamma_{1,a,a+1} \leq 1/12+\eps$ if 
$$
a > (K_1K_2)^{K_1},
$$
where $K_1,K_2$ come from the regularity lemma:
$$
K_1 = \eps^{-C}, \ \ K_2 = \exp(\eps^{-C}).
$$
So his result implies that
$$
\gamma_{1,a,a+1} \leq \frac{1}{12} + O\Big(\frac{1}{(\log a)^c}\Big).
$$
\end{remark}

\section{Proof outline of Theorem \ref{thm:main}}\label{sec:pf-outline}

In this section we sketch the proof of Theorem \ref{thm:main}. Let $A \subset \Z$ be a finite subset. The first step is to use the structure theorem in \cite{green2008maximal} to reduce to the case when $A$ is a dense subset of an interval, summarized in the following proposition.

\begin{proposition}\label{prop:main-dense-case}
Let $a_1,a_2,a_3$ be fixed nonzero integers with $\gcd(a_1,a_2,a_3)=1$. Let $I \subset \Z$ be an interval of length $N$ and let $A \subset I$ be a subset. For any $\eps \in (0,1/2)$, there exists a positive integer $q = q(\eps)$, all of whose prime factors divide $a_1a_2a_3$, such that
$$
T_{a_1,a_2,a_3}(A) \leq \gamma_{a_1,a_2,a_3}(q)|A|^2 + \eps N^2,
$$
provided that $N$ is sufficiently large in terms of $\eps$.
\end{proposition}

The deduction of Theorem \ref{thm:main} from Proposition \ref{prop:main-dense-case} is carried out in Section \ref{sec:proof1}.

To prove Proposition \ref{prop:main-dense-case}, we start with applying (the abelian case of) the arithmetic regularity lemma in \cite{GreenTao-regularity} (see also the note \cite{Eberhard-note}) to appoximate $1_A$ by a function of the form
$$
n\mapsto F(n\Mod{q}, \frac{n}{N}, n\theta)
$$
for some positive integer $q$, some ``highly irrational" $\theta \in (\R/\Z)^d$ for some positive integer $d$, and some Lipschitz function $F$ on $X := \Z/q\Z \times \R \times (\R/\Z)^d$. For example, a set $A\subset \Z$ of the form
$$
A = \bigcup_{r=1}^q \Big\{n \equiv r\Mod{q}: \Big|\frac{n}{N}\Big| \leq \frac{g(r)}{2}\Big\}
$$
should be modeled by the function $F = 1_{\wt{A}}$ where $\wt{A} \subset X$ is defined by
\begin{equation}\label{eq:wtA}
\wt{A} = \bigcup_{r=1}^q \left(\{r\} \times [-g(r)/2,g(r)/2] \times (\R/\Z)^d\right)
\end{equation}
Using this approximation, we will reduce the problem of maximizing $T_{a_1,a_2,a_3}(A)$ to the corresponding problem of maximizing $\wt{T}_{a_1,a_2,a_3}(\wt{A})$, where $\wt{A} \subset X$ and $\wt{T}(\wt{A})$ ``counts" the solutions to $a_1y_1+a_2y_2+a_3y_3 = 0$ with $y_1,y_2,y_3 \in \wt{A}$.

 The counting operator $\wt{T}_{a_1,a_2,a_3}$ is properly defined as follows. Note that there is an obvious metric on $X$, namely the product metric of the discrete metric on $\Z/q\Z$ and the Euclidean metrics on $\R$ and $(\R/\Z)^d$, and also an obvious measure on $X$ which we always denote by $\mu$, which is the product of the probability measure on $\Z/q\Z$ (which assigns mass $1/q$ to each element of $\Z/q\Z$) and the Lebesgue measures on $\R$ and $(\R/\Z)^d$.

\begin{definition}\label{def:Ttilde}
Let $a_1,a_2,a_3$ be nonzero integers.
\begin{enumerate}
\item  For measurable functions $F_1,F_2,F_3$ on $(\R/\Z)^d$, we define
$$
\wt{T}_{a_1,a_2,a_3}(F_1,F_2,F_3) := \int F_1(y_1) F_2(y_2) \Big(\E_{y_3: a_1y_1+a_2y_2+a_3y_3=0} F_3(y_3)\Big) d\mu(y_1) d\mu(y_2).
$$
\item For finitely supported measurable function $F_1,F_2,F_3$ on $\R \times (\R/\Z)^d$, we define
$$
\wt{T}_{a_1,a_2,a_3}(F_1,F_2,F_3) = \frac{1}{|a_3|}\int F_1(y_1) F_2(y_2) \Big(\E_{y_3: a_1y_1+a_2y_2+a_3y_3=0} F_3(y_3)\Big) d\mu(y_1) d\mu(y_2),
$$
\item 
Suppose that $\gcd(a_1,a_2,a_3)=1$. For finitely supported measurable functions $F_1,F_2,F_3$ on $X = \Z/q\Z \times \R \times (\R/\Z)^d$, we define
$$
\wt{T}_{a_1,a_2,a_3}(F_1,F_2,F_3) := \frac{1}{q^2} \sum_{\substack{r_1,r_2,r_3\Mod{q} \\ a_1r_1+a_2r_2+a_3r_3\equiv 0\Mod{q}}} \wt{T}_{a_1,a_2,a_3}(F_1(r_1,\cdot), F_2(r_2,\cdot), F_3(r_3,\cdot))
$$
\end{enumerate}
If $F_1=F_2=F_3 = F$ we write $\wt{T}_{a_1,a_2,a_3}(F)$ for $\wt{T}_{a_1,a_2,a_3}(F_1,F_2,F_3)$.
\end{definition}

In (1) and (2), given $y_1,y_2$, there are $|a_3|^d$ values of $y_3$ in the inner average. 
In (3), the number of solutions to $a_1r_1+a_2r_2+a_3r_3\equiv 0\Mod{q}$ is $q^2$ (since $\gcd(a_1,a_2,a_3)=1$), and thus there are $q^2$ terms in the sum. One would expect $\wt{T}_{a_1,a_2,a_3}(F)$ to be symmetric in $a_1,a_2,a_3$. This is not obvious from the definition, but will be confirmed by Lemma \ref{lem:TF-TFa} below.

In Section \ref{sec:proof2} we will use the arithmetic regularity lemma to prove the following proposition, allowing us to pass from counting in $\Z$ to counting in $X$.

\begin{proposition}\label{prop:proof2}
Let $a_1,a_2,a_3$ be fixed nonzero integers with $\gcd(a_1,a_2,a_3)=1$. Let $I \subset \Z$ be an interval of length $N$ and let $A \subset I$ be a subset. For any $\eps \in (0,1/2)$, there exist positive integers $q = q(\eps)$, $d = d(\eps)$, and a finitely supported measurable function $F$ on $X := \Z/q\Z\times \R \times (\R/\Z)^d$ taking values in $[0,1]$ with $\int F = |A|/N + O(\eps)$, such that
$$
T_{a_1,a_2,a_3}(A) = (\wt{T}_{a_1,a_2,a_3}(F) + O(\eps))N^2,
$$
provided that $N$ is sufficiently large in terms of $\eps$.
\end{proposition}

In Section \ref{sec:proof3} we will study the problem of maximizing $\wt{T}_{a_1,a_2,a_3}(F)$ and show that the maximizer must take the form $F = 1_{\wt{A}}$ for some $\wt{A}$ of the form \eqref{eq:wtA}, thus connecting the maximum value of $\wt{T}_{a_1,a_2,a_3}(F)$ with $\gamma_{a_1,a_2,a_3}(q)$.

\begin{proposition}\label{prop:proof3}
Let $a_1,a_2,a_3$ be fixed nonzero integers with $\gcd(a_1,a_2,a_3)=1$. Let $X = \Z/q\Z \times \R \times (\R/\Z)^d$ for some positive integers $q,d$. Then for any finitely supported measurable function $F: X \rightarrow [0,1]$ with $\delta = \int F$, we have
$$
\wt{T}_{a_1,a_2,a_3}(F) \leq \gamma_{a_1,a_2,a_3}(q) \cdot \delta^2.
$$
\end{proposition}

Finally, in Section \ref{sec:proof3} we will also prove the following lemma which allows us to restrict to those $q$ all of whose prime factors divide $a_1a_2a_3$ in the computation of $\gamma_{a_1,a_2,a_3}(q)$.

\begin{lemma}\label{lem:q-divide-a1a2a3}
Let $a_1,a_2,a_3$ be fixed nonzero integers with $\gcd(a_1,a_2,a_3)=1$. Let $q,m$ be positive integers with $(q,m)=1$ and $(m,a_1a_2a_3)=1$. Then $\gamma_{a_1,a_2,a_3}(qm) = \gamma_{a_1,a_2,a_3}(q)$.
\end{lemma}

Clearly Proposition \ref{prop:main-dense-case} follows by combining Propositions \ref{prop:proof2} and \ref{prop:proof3} and using Lemma~\ref{lem:q-divide-a1a2a3}.

\section{Reducing to the dense case}\label{sec:proof1}

In this section we deduce Theorem \ref{thm:main} assuming Proposition \ref{prop:main-dense-case}. Our goal is to prove that for every $\eps > 0$, there exists a positive integer $q = q(\eps)$, all of whose prime factors divide $|a_1a_2a_3|$, such that we have
$$
T_{a_1,a_2,a_3}(A) \leq (\gamma_{a_1,a_2,a_3}(q) + O(\eps)) |A|^2
$$
for all finite subsets $A \subset \Z$ with $|A|$ sufficieintly large in terms of $\eps$.
Let $\eps'>0$ be a constant sufficiently small in terms of $\eps$. We apply Proposition 3.2 from \cite{green2008maximal} to obtain a partition $A = A_1 \cup A_2 \cup \cdots \cup A_n \cup A_0$ satisfying the following conditions:
\begin{itemize}
\item (Components are large) $|A_i| \gg_{\eps} |A|$ for each $1 \leq i \leq n$;
\item (Components are structured) $|A_i+A_i| \ll_{\eps,\eps'} |A_i|$ for each $1 \leq i \leq n$;
\item (Different components do not communicate) $E(\lambda_i\cdot A_i, \lambda_j \cdot A_j) \leq \eps' |A_i|^{3/2}|A_j|^{3/2}$ for all $1 \leq i < j \leq n$ and $\lambda_i,\lambda_j \in \{a_1, a_2, a_3\}$;
\item (Small noise term) $E(\lambda_0\cdot A_0, \lambda \cdot A) \leq \eps^2 |A|^3$ for all $\lambda_0,\lambda \in \{a_1, a_2, a_3\}$.
\end{itemize}
In particular we have $n \ll_{\eps} 1$.

\begin{lemma}\label{lem:TA-TAi}
Let the notations and assumptions be as above. We have
$$
T_{a_1,a_2,a_3}(A) = \sum_{i=1}^n T_{a_1,a_2,a_3}(A_i) + O(\eps |A|^2).
$$
\end{lemma}

\begin{proof}
Decompose $T_{a_1,a_2,a_3}(A)$ into terms of the form $T_{a_1,a_2,a_3}(A_i, A_j, A_k)$. The total contribution from those terms with $i=0$ is
$$
\sum_{j,k} T_{a_1,a_2,a_3}(A_0, A_j, A_k) \leq T_{a_1,a_2,a_3}(A_0, A, A) \leq E(a_1\cdot A_0, a_2 \cdot A)^{1/2}|A|^{1/2} \leq \eps |A|^2,
$$
where we used \cite[Lemma 2.3]{Aa} to bound $T_{a_1,a_2,a_3}(A_0, A, A)$. The same bound also applies to the total contribution from those terms with $j=0$ or $k=0$. Hence
$$
T_{a_1,a_2,a_3}(A) = \sum_{1 \leq i,j,k \leq n} T_{a_1,a_2,a_3}(A_i,A_j,A_k) + O(\eps |A|^2).
$$
If $1 \leq i,j,k \leq n$ and $i,j,k$ are not all equal, say $i \neq j$, then
$$
T_{a_1,a_2,a_3}(A_i,A_j,A_k) \leq E(a_1\cdot A_i, a_2\cdot A_j)^{1/2} |A_k|^{1/2} \leq \eps' |A|^2,
$$
where we again used  \cite[Lemma 2.3]{Aa}. It follows that
$$
T_{a_1,a_2,a_3}(A) = \sum_{i=1}^nT_{a_1,a_2,a_3}(A_i) + O(n^3\eps' |A|^2) + O(\eps |A|^2).
$$
Since $n \ll_{\eps}1$, the desired conclusion follows if we choose $\eps' < \eps/n^3$.
\end{proof}

The following lemma allows us to model each $A_i$ for $1 \leq i \leq n$, which has small doubling, by a dense subset of an interval. 

\begin{lemma}\label{lem:rectification}
Let $B \subset \Z$ be a finite subset with $0 \in B$ and $|B+B| \leq K|B|$ for some $K \geq 2$. Then there exists an interval $I \subset \Z$ with $|I| \ll_K |B|$ and a subset $\wt{B} \subset I$ with $|\wt{B}| = |B|$ such that $T_{a_1,a_2,a_3}(\wt{B}) = T_{a_1,a_2,a_3}(B)$.
\end{lemma}

\begin{proof}
Set $t = |a_1| + |a_2| + |a_3|$. By a Freiman modeling lemma (or rectification, such as Theorem 1.4 in \cite{GreenRuzsa}), there exists a Freiman $t$-isomorphism $\phi: B \rightarrow \wt{B}$ for some subset $\wt{B}$ of an interval $I$ with $|I| \ll_K |B|$. Thus
$$
\sum_{i=1}^3a_ix_i= \sum_{i=1}^3a_ix_i' \Longleftrightarrow \sum_{i=1}^3a_i\phi(x_i) = \sum_{i=1}^3a_i\phi(x_i')
$$
for $x_1,x_2,x_3,x_1',x_2',x_3' \in B$. By translating we may assume that $\phi(0) = 0$. Setting $x_1'=x_2'=x_3' = 0$ we have
$$
a_1x_1 + a_2x_2 + a_3x_3 = 0 \Longleftrightarrow a_1\phi(x_1) + a_2\phi(x_2) + a_3\phi(x_3) = 0
$$
for $x_1,x_2,x_3 \in B$. This implies that $T_{a_1,a_2,a_3}(\wt{B}) = T_{a_1,a_2,a_3}(B)$ as desired.    
\end{proof}

For each $1 \leq i \leq n$, by Lemma \ref{lem:rectification} applied to $A_i \cup \{0\}$, we obtain an interval $I_i \subset \Z$ with $|I_i| \ll_{\eps,\eps'} |A_i|$ and a subset $\wt{A}_i \subset I_i$ with $|\wt{A}_i| \leq |A_i|+1$ such that $T_{a_1,a_2,a_3}(A_i) \leq T_{a_1,a_2,a_3}(\wt{A}_i)$. Let $\eps'' > 0$ be a constant sufficiently small in terms of $\eps, \eps'$. By Proposition \ref{prop:main-dense-case} applied to $A_i'$, there exists a positive integer $q = q(\eps'')$ (which we can take to be independent of $i$), all of whose prime factors divide $a_1a_2a_3$, such that
$$
T_{a_1,a_2,a_3}(A_i) \leq \gamma_{a_1,a_2,a_3}(q)|A_i|^2 + \eps'' |I_i|^2 \leq (\gamma_{a_1,a_2,a_3}(q) + \eps)|A_i|^2,
$$
where the second inequality can be guaranteed once $\eps''$ is small enough in terms of $\eps, \eps'$. By Lemma \ref{lem:TA-TAi}, it follows that
$$
T_{a_1,a_2,a_3}(A) \leq \sum_{i=1}^n (\gamma_{a_1,a_2,a_3}(q) + \eps)|A_i|^2  + O(\eps |A|^2) \leq (\gamma_{a_1,a_2,a_3}(q) + O(\eps))|A|^2.
$$
This concludes the proof.

\section{Passing from $\Z$ to $X$ via the regularity lemma}\label{sec:proof2}

In this section we prove Proposition \ref{prop:proof2}. Throughout this section, we fix nonzero integers $a_1,a_2,a_3$ with $\gcd(a_1,a_2,a_3)=1$. We abbreviate $T,\wt{T}$ for $T_{a_1,a_2,a_3},\wt{T}_{a_1,a_2,a_3}$, respectively. We will apply the abelian case of the arithmetic regularity lemma in the following form. 

\begin{theorem}[Arithmetic regularity lemma]
Let $\eps \in (0,1/2)$ and let $\CF:\R_{>0}\rightarrow\R_{>0}$ be a growth function. Let $I \subset \Z$ be an interval of length $N$, and let $f: I\rightarrow [0,1]$ be a function. Then there exists $M \ll_{\eps,\CF}1$ such that one can decompose $f$ into
$$
f = f_{\str} + f_{\sml} + f_{\unf},
$$
where $f_{\str}, f_{\sml}, f_{\unf}: I \rightarrow [-1,1]$ satisfy the following properties:
\begin{enumerate}
\item $f_{\str}(n) = F(\pi(n))$, where $\pi: \Z\rightarrow X:= \Z/q\Z \times \R \times (\R/\Z)^d$ for some $q,d \leq M$ is defined by
$$
\pi(n) = \Big(n\Mod{q}, \frac{n}{N}, n\theta\Big)
$$
for some $\theta \in (\R/\Z)^d$ which is $(\CF(M), N)$-irrational, and $F: X\rightarrow [0,1]$ has Lipschitz norm at most $M$ and is supported on $\Z/q\Z \times (N^{-1}\cdot I) \times (\R/\Z)^d$.
\item $\|f_{\sml}\|_{\ell^2(I)} := (\E_{n \in I}|f_{\sml}(n)|^2)^{1/2}  \leq \eps$.
\item $\|f_{\unf}\|_{U^2(I)} \leq 1/\CF(M)$.
\end{enumerate}
\end{theorem}

\begin{remark}
When $I = [N]$, this is precisely \cite[Lemma A.2]{EberhardGreenManners}. For general $I$, it follows by applying \cite[Lemma A.2]{EberhardGreenManners} to a suitable translate of $f$. The definition of $(\CF(M),N)$-irrationality of $\theta$ can be found in \cite[Definition A.1]{EberhardGreenManners}. The definition of the Gowers $U^2(I)$ norm can be found in \cite[Definition A.7]{EberhardGreenManners} when $I = [N]$. For general $I$, there is an obvious definition of $\|f\|_{U^2(I)}$ in terms of the $U^2([N])$ norm of a shifted version of $f$. For a fuller discussion of Gowers norms, see \cite[Section 11.1]{TaoVu-book}.
\end{remark}

Let $\CF:\R_{>0}\rightarrow\R_{>0}$ be an increasing function growing sufficiently rapidly in terms of $\eps$. By the arithmetic regularity lemma applied to $1_A$, there exists $M \ll_{\eps,\CF}1$ and $q,d \leq M$ such that
$$
1_A = f_{\str} + f_{\sml} + f_{\unf}
$$ 
for $n \in I$, where $f_{\str} = F\circ \pi$ and $\pi: \Z\rightarrow X$ is defined by
$$
\pi(n) = \Big(n\Mod{q}, \frac{n}{N}, n\theta\Big)
$$
for some $\theta \in (\R/\Z)^d$ which is $(\CF(M), N)$-irrational, and $F: X\rightarrow [0,1]$ has Lipschitz norm at most $M$ and is supported on $\Z/q\Z \times (N^{-1}\cdot I)\times (\R/\Z)^d$. Moreover, $f_{\sml}$ and $f_{\unf}$ take values in $[-1,1]$ and satisfy 
$$
\|f_{\sml}\|_{\ell^2(I)} \leq \eps, \ \ \|f_{\unf}\|_{U^2(I)} \leq 1/\CF(M).
$$
By choosing $\CF$ to grow rapidly enough in terms of $\eps$, we may ensure that $\CF(M)$ is sufficiently large in terms of $\eps$ and $M$. We may also assume that $N$ is sufficiently large in terms of $\eps$ and $M$, since otherwise the conclusion of Proposition \ref{prop:proof2} is trivial. Since $\theta$ is $(\CF(M),N)$-irrational, it follows from \cite[Lemma A.4]{EberhardGreenManners} that
$$
\int F = \E_{n \in I} f_{\str}(n) + O(\eps).
$$
Since $\E_{n \in I}|f_{\sml}(n)| \leq \|f_{\sml}\|_{\ell^2(I)} \leq \eps$ and $\E_{n \in I} |f_{\unf}(n)| \ll \|f_{\unf}\|_{U^2(I)} \leq \eps$ (by \cite[Lemma A.8]{EberhardGreenManners}), we have
$$
\int F = \E_{n \in I} 1_A(n) + O(\eps) = \frac{|A|}{N} + O(\eps).
$$
We decompose $T(A)$ into a sum of nine terms, each of the form $T(f_1, f_2, f_3)$ where $f_i$ is either $f_{\str} = F\circ\pi$, $f_{\sml}$, or $f_{\unf}$.

\begin{lemma}\label{lem:l2-u2-control}
Let the notations and assumptions be as above. We have
$$
|T(f_1,f_2,f_3)| \leq \|f_i\|_{\ell^2(I)}\cdot N^2 \text{ and }|T(f_1,f_2,f_3)| \ll \|f_i\|_{U^2(I)}\cdot N^2
$$
for each $i\in\{1,2,3\}$.
\end{lemma}

\begin{proof}
Without loss of generality we may assume that $i=1$. For the bound in terms of the $\ell^2(I)$ norm, we use the triangle inequality, the fact that $f_2,f_3$ are bounded pointwise by $1$, and the Cauchy-Schwarz inequality to obtain
$$
|T(f_1,f_2,f_3)| \leq \sum_{x_1,x_2 \in I} |f_1(x_1)| \leq  N^2 \Big(\E_{x_1 \in I}|f_1(x_1)|^2\Big)^{1/2} = \|f_1\|_{\ell^2(I)} \cdot N^2.
$$
The bound in terms of the $U^2(I)$ norm follows from an instance of the generalized von-Neumann theorem (see \cite[Lemma 11.4]{TaoVu-book}), after embedding $I$ into an appropriate cyclic group.    
\end{proof}

It follows from Lemma \ref{lem:l2-u2-control} that $T(A) = T(f_{\str}) + O(\eps N^2)$. Thus the proof of Proposition \ref{prop:proof2} is completed once we establish the following lemma connecting $T(f_{\str})$ with $\wt{T}(F)$.

\begin{lemma}\label{lem:T-Ttilde}
Let $\theta \in (\R/\Z)^d$ be $(A,N)$-irrational, and let $F_1,F_2,F_3: X \rightarrow [0,1]$ (where $X = \Z/q\Z \times \R \times (\R/\Z)^d$)  be functions with Lipschitz norm at most $M$ supported on $\Z/q\Z \times I \times (\R/\Z)^d$ for some interval $I \subset \R$ of length $1$. Let $f_i : \Z\rightarrow [0,1]$ be defined by $f_i(n) = F_i(\pi(n)) = F_i(n\Mod{q}, n/N, n\theta)$. Then, provided that $A,N$ are sufficiently large in terms of $q,d,M,\eps$, we have
\begin{equation}\label{eq:T-Ttilde}
T(f_1,f_2,f_3) = (\wt{T}(F_1,F_2,F_3) + O(\eps))N^2.
\end{equation}
\end{lemma}

\begin{proof}
In the case when $|a_1|=|a_2|=|a_3|=1$, this is essentially the conclusion in \cite[Lemma A.5]{EberhardGreenManners} when $F_1=F_2=F_3$. For general $a_1,a_2,a_3$, we proceed as follows.  Let $J = \lceil M/\eps\rceil$ and partition $I$ into $J$ subintervals $I_1,\cdots,I_J$ of length $1/J$. For $r\in\Z/q\Z$ and $1 \leq j \leq J$, let $X_{r,j} = \{r\} \times I_j \times (\R/\Z)^d$ and $P_{r,j} = \pi^{-1}(X_{r,j})$. Then $P_{r,j}$ is an arithmetic progression of step $q$ with $|P_{r,j}|\sim N/(qJ)$. We may decompose
$$
T(f_1,f_2,f_3) = \sum_{r_1,r_2,r_3 \in \Z/q\Z} \sum_{1 \leq j_1,j_2,j_3 \leq J} T(f_11_{P_{r_1,j_1}}, f_21_{P_{r_2,j_2}}, f_31_{P_{r_3,j_3}}).
$$
Note that if the summand is nonzero then there exists $n_i \in P_{r_i,j_i}$ for $i \in \{1,2,3\}$ satisfying $a_1n_1+a_2n_2+a_3n_3=0$. This implies that $a_1r_1+a_2r_2+a_3r_3\equiv 0\Mod{q}$ and $a_1I_{j_1}+a_2I_{j_2}+a_3I_{j_3} \subset [-O(1/J), O(1/J)]$. It follows that there are $O((qJ)^2)$ nonzero terms in this decomposition. Similarly, we can decompose
$$
\wt{T}(F_1,F_2,F_3) =  \sum_{r_1,r_2,r_3 \in \Z/q\Z} \sum_{1 \leq j_1,j_2,j_3 \leq J} \wt{T}(F_11_{X_{r_1,j_1}}, F_21_{X_{r_2,j_2}}, F_31_{X_{r_3,j_3}}),
$$
and there are $O((qJ)^2)$ nonzero terms in this decomposition. To prove \eqref{eq:T-Ttilde}, it suffices to prove for every $r_1,r_2,r_3 \in \Z/q\Z$ and $1 \leq j_1,j_2,j_3 \leq J$ that
\begin{equation}\label{eq:T-Ttilde2}
T(f_11_{P_{r_1,j_1}}, f_21_{P_{r_2,j_2}}, f_31_{P_{r_3,j_3}}) = \wt{T}(F_11_{X_{r_1,j_1}}, F_21_{X_{r_2,j_2}}, F_31_{X_{r_3,j_3}}) N^2 + O(\eps |P_{r_1,j_1}|^2).
\end{equation}

For each $r\in\Z/q\Z$ and $1 \leq j \leq J$, pick an arbitrary element $s_j \in I_j$ and define for $i \in \{1,2,3\}$ the function $F_i^{(r,j)}: (\R/\Z)^d\rightarrow \C$ by $F_i^{(r,j)}(t) = F_i(r, s_j, t)$ for $t \in (\R/\Z)^d$. Since $F_i$ has Lipschitz norm at most $M$, $F_i^{(r,j)}$ has Lipschitz norm at most $M$ for each $r,j$. Moreover, for all $x = (r,s,t)\in X_{r,j}$ we have
$$
|F_i(x) - F_i^{(r,j)}(t)| = |F(r,s,t) - F(r,s_j,t)| \leq M\cdot |s-s_j| \leq \frac{M}{J} \ll \eps
$$
by our choice of $J$. For $n \in P_{r,j}$, we have $\pi(n) \in X_{r,j}$ and hence
$$
|f_i(n) - F_i^{(r,j)}(n\theta)|  = |F_i(\pi(n)) - F_i^{(r,j)}(n\theta)| \ll \eps.
$$
It follows that we can replace $f_i1_{P_{r_i,j_i}}$ and $F_i1_{X_{r_i,j_i}}$ in \eqref{eq:T-Ttilde2} by $F_i^{(r_i,j_i)}(\cdot \theta)1_{P_{r_i,j_i}}$ and $F_i^{(r_i,j_i)}1_{X_{r_i,j_i}}$, respectively, at the cost of an acceptable error. This reduces matters to proving
\begin{multline}\label{eq:T-Ttilde3}
T(F_1^{(r_1,j_1)}(\cdot\theta)1_{P_{r_1,j_1}}, F_2^{(r_2,j_2)}(\cdot\theta)1_{P_{r_2,j_2}}, F_3^{(r_3,j_3)}(\cdot\theta)1_{P_{r_3,j_3}}) \\
 =\wt{T}(F_1^{(r_1,j_1)}1_{X_{r_1,j_1}}, F_2^{(r_2,j_2)}1_{X_{r_2,j_2}}, F_3^{(r_3,j_3)}1_{X_{r_3,j_3}}) N^2 + O(\eps |P_{r_1,j_1}|^2).
\end{multline}
Since $F_i^{(r_i,j_i)}(x)$ depends only on the third component of $x$ and $1_{X_{r_i,j_i}}(x)$ depends only on the first and second component of $x$, the $\wt{T}$ term on the right-hand side above can be factored as
$$
\wt{T}(F_1^{(r_1,j_1)}, F_2^{(r_2,j_2)}, F_3^{(r_3,j_3)}) \cdot \wt{T}(X_{r_1,j_1}, X_{r_2,j_2}, X_{r_3,j_3}).
$$
By the definition of $\wt{T}$, we have for $a_1r_1 + a_2r_2 +a_3r_3 \equiv 0\Mod{q}$ that
$$
\wt{T}(X_{r_1,j_1}, X_{r_2,j_2}, X_{r_3,j_3}) = \frac{1}{q^2} \wt{T}(I_{j_1}, I_{j_2}, I_{j_3}).
$$
On the other hand, since $P_{r_i,j_i}$ is the congruence class $r_i\Mod{q}$ in the interval $N\cdot I_{j_i}$ which has length $\sim N/J$, we have
$$
T(P_{r_1,j_1}, P_{r_2,j_2}, P_{r_3,j_3}) \sim \frac{1}{q^2} \wt{T}(N\cdot I_{j_1}, N \cdot I_{j_2}, N\cdot I_{j_3})= \frac{N^2}{q^2} \wt{T}(I_{j_1}, I_{j_2}, I_{j_3}) .
$$
It follows that
$$
T(P_{r_1,j_1},P_{r_2,j_2},P_{r_3,j_3}) \sim \wt{T}(X_{r_1,j_1}, X_{r_2,j_2}, X_{r_3,j_3})N^2.
$$
In view of the above, the desired estimate \eqref{eq:T-Ttilde3} follows once we apply the following lemma to the left-hand side, completing the proof of Proposition \ref{prop:proof2}.
\end{proof}

\begin{lemma}\label{lem:equi-dist}
Let $\theta \in (\R/\Z)^d$ be $(A,N)$-irrational, let $F_1,F_2,F_3: (\R/\Z)^d \rightarrow \C$ be functions with Lipschitz norm at most $M$, and let $P_1,P_2,P_3 \subset \Z$ be arithmetic progressions of length between $\eta N$ and $2\eta N$. Then, provided that $A,N$ are sufficiently large in terms of $d, M, \eta,\eps$, we have
$$
 \sum_{\substack{n_1 \in P_1,n_2\in P_2,n_3\in P_3 \\ a_1n_1+a_2n_2+a_3n_3=0}} F_1(n_1\theta) F_2(n_2\theta) F_3(n_3\theta)
 = \wt{T}(F_1, F_2, F_3)\cdot T(P_1,P_2,P_3) + O(\eps |P_1|^2).
 $$
\end{lemma}

\begin{proof}
This is a generalization of \cite[Lemma A.3]{EberhardGreenManners} and we follow the proof there. Write $e(y) := \exp(2\pi iy)$. Approximate each $F_i$ by a truncated Fourier expansion:
$$
F_i(x) = \sum_{m \in \Z^d} c_i(m)e(m\cdot x) + O(\eps),
$$
where $c_i(m)$ is supported on $|m| \ll_{d,M,\eps}1$, $|c_i(m)| \ll_{d,M}1$, and $c_i(0) = \int F_i$. Denote the left-hand side by $T$. We have
$$
T = \sum_{m_1,m_2,m_3 \in \Z^d}c_1(m_1)c_2(m_2)c_3(m_3) \sum_{\substack{n_1 \in P_1,n_2\in P_2,n_3\in P_3 \\ a_1n_1+a_2n_2+a_3n_3=0}} e((n_1m_1+n_2m_2+n_3m_3)\cdot \theta) + O(\eps |P_1|^2).
$$
The inner sum above can be rewritten as
$$
\sum_{\substack{a_3 \mid a_1n_1+a_2n_2 \\ n_1 \in P_1, n_2 \in P_2, -(a_1n_1+a_2n_2)/a_3 \in P_3}} e\Big(n_1\Big(m_1-\frac{a_1}{a_3}m_3\Big)\cdot\theta + n_2\Big(m_2-\frac{a_2}{a_3}m_3\Big)\cdot\theta\Big).
$$
Let $\eps' > 0$ be sufficiently small in terms of $d,M,\eps$. Since $\theta$ is $(A,N)$-irrational for some sufficiently large $A$, the sum above is $O(\eps' |P_1|^2)$ unless 
$$
m_1 - \frac{a_1}{a_3}m_3 = m_2 - \frac{a_2}{a_3}m_3 = 0,
$$
in which case $(m_1,m_2,m_3) = (a_1m, a_2m, a_3m)$ for some $m \in \Z^d$. Hence
$$
T = \sum_{m \in \Z^d} c_1(a_1m)c_2(a_2m)c_3(a_3m) \Big(\sum_{\substack{n_1 \in P_1,n_2\in P_2,n_3\in P_3 \\ a_1n_1+a_2n_2+a_3n_3=0}}1\Big) + O_{d,M,\eps}(\eps' |P_1|^2) + O(\eps |P_1|^2).
$$
By choosing $\eps' > 0$ to be small enough in terms of $d,M,\eps$, we may ensure that the error terms above are $O(\eps |P_1|^2)$, and we have
\begin{equation}\label{eq:equi-dist1}
T = T(P_1,P_2,P_3) \sum_{m \in \Z^d} c_1(a_1m)c_2(a_2m)c_3(a_3m) + O(\eps |P_1|^2).
\end{equation}
On the other hand, since
$$
\wt{T}(F_1,F_2,F_3) = \int F_1(x_1) F_2(x_2) \Big(\E_{x_3:a_1x_1+a_2x_2+a_3x_3=0} F_3(x_3)\Big) d\mu(x_1) d\mu(x_2),
$$
after plugging in the Fourier expansions we can rewrite the above as (at the cost of an error of $O(\eps)$)
$$
\sum_{m_1,m_2,m_3 \in \Z^d} c_1(m_1)c_2(m_2)c_3(m_3) \int e(m_1\cdot x_1+m_2\cdot x_2) \Big(\E_{x_3:a_1x_1+a_2x_2+a_3x_3=0} e(m_3\cdot x_3)\Big) d\mu(x_1) d\mu(x_2).
$$
The inner average over $x_3$ above vanishes unless $m_3 = a_3m$ for some $m \in \Z^d$, in which case the average is $e(m\cdot (-a_1x_1-a_2x_2))$ and the integral over $x_1,x_2$ becomes 
$$
\int e((m_1-a_1m)\cdot x_1+(m_2-a_2m)\cdot x_2) d\mu(x_1) d\mu(x_2).
$$
This integral vanishes unless $m_1=a_1m$ and $m_2 = a_2m$. It follows that
$$
\wt{T}(F_1,F_2,F_3) = \sum_{m \in \Z^d} c_1(a_1m)c_2(a_2m)c_3(a_3m) + O(\eps).
$$
Combining this with \eqref{eq:equi-dist1} completes the proof.
\end{proof}

\section{Maximizing the number of solutions in $X$}\label{sec:proof3}

In this section we prove Proposition \ref{prop:proof3} and Lemma \ref{lem:q-divide-a1a2a3}.  For $r\Mod{q}$, write $F_r: \R \times (\R/\Z)^d \rightarrow [0,1]$ for the function defined by $F_r(s,t) = F(r,s,t)$. Then $\E_{r\Mod{q}} \int F_r = \int F = \delta$ and 
$$
\wt{T}_{a_1,a_2,a_3}(F) = \frac{1}{q^2}\sum_{\substack{r_1,r_2,r_3\Mod{q} \\ a_1r_1+a_2r_2+a_3r_3\equiv 0\Mod{q}}} \wt{T}_{a_1,a_2,a_3}(F_{r_1},F_{r_2}, F_{r_3}).
$$

\begin{lemma}\label{lem:TF-TFa}
For finitely supported measurable functions $F_1, F_2, F_3: \R \times (\R/\Z)^d \rightarrow \C$, we have
$$
\wt{T}_{a_1,a_2,a_3}(F_1,F_2,F_3) = \frac{1}{|a_1a_2a_3|}\wt{T}_{1,1,1}(F_1^{(a_1)}, F_2^{(a_2)}, F_3^{(a_3)}),
$$ 
where $F^{(a)}: \R \times (\R/\Z)^d \rightarrow \C$ is defined by
$$
F^{(a)}(y) = \E_{x: y=ax} F(x).
$$
Moreover, we have $\int F^{(a)} = |a|\int F$.
\end{lemma}

\begin{proof}
Without loss of generality, we may assume that $a_1,a_2,a_3 > 0$. By definition, we have
$$
\wt{T}_{a_1,a_2,a_3}(F_1,F_2,F_3) = \frac{1}{a_3}\int F_1(x_1) F_2(x_2) F_3^{(a_3)}(-a_1x_1-a_2x_2) d\mu(x_1) d\mu(x_2).
$$
View $x_1,x_2$ as elements in $\R\times [0,1)^d$ in the natural way. By making a change of variables $y_i = a_ix_i$, we can rewrite the integral above as
$$
\frac{1}{(a_1a_2)^{d+1}}\int_{\R \times [0,a_1)^d} \int_{\R \times [0,a_2)^d} F_1\Big(\frac{y_1}{a_1}\Big) F_2\Big(\frac{y_2}{a_2}\Big) F_3^{(a_3)}(-y_1-y_2) d\mu(y_1) d\mu(y_2).
$$
After dividing $[0,a_i)^d$ into cubes of the form $k_i + [0,1)^d$ for some $k_i\in \{0,1,\cdots,a_i-1\}^d$, we can rewrite the above integral as
$$
\frac{1}{(a_1a_2)^{d+1}} \sum_{k_1,k_2} \int_{\R\times [0,1)^d} \int_{\R\times [0,1)^d} F_1\Big(\frac{y_1 + (0,k_1)}{a_1}\Big) F_2\Big(\frac{y_2 + (0,k_2)}{a_2}\Big) F_3^{(a_3)}(-y_1-y_2) d\mu(y_1) d\mu(y_2).
$$
Moving the sum over $k_1,k_2$ inside leads to
$$
\wt{T}_{a_1,a_2,a_3}(F_1,F_2,F_3) = \frac{1}{a_1a_2a_3} \int F_1^{(a_1)}(y_1) F_2^{(a_2)}(y_2) F_3^{(a_3)}(-y_1-y_2) d\mu(y_1) d\mu(y_2).
$$
This completes the proof of the identity relating $\wt{T}_{a_1,a_2,a_3}$ with $\wt{T}_{1,1,1}$. In particular, we have
$$
\wt{T}_{a,1,1}(F, F_0, F_0) = \frac{1}{|a|}\wt{T}_{1,1,1}(F^{(a)}, F_0, F_0).
$$
Take $F_0$ to be the characteristic function of $[-M, M] \times (\R/\Z)^d$ for some sufficiently large $M$. Dividing both sides by $2M$ the left-hand side and the right-hand side above tend to $\int F$ and $|a|^{-1}\int F^{(a)}$ as $M\rightarrow \infty$. This implies that $\int F^{(a)} = |a|\int F$.
\end{proof}

\begin{lemma}\label{lem:max-T111}
For finitely supported measurable functions $F_1, F_2, F_3: \R \times (\R/\Z)^d \rightarrow [0,1]$ with $\delta_i = \int F_i$ for $i \in \{1,2,3\}$, we have
$$
\wt{T}_{1,1,1}(F_1, F_2, F_3) \leq G(\delta_1, \delta_2, \delta_3).
$$
\end{lemma}

\begin{proof}
We will prove that 
$$
\wt{T}_{1,1,1}(F_1, F_2, F_3) \leq G(\delta_1, \delta_2, \delta_3) + O(\eps)
$$
for every $\eps > 0$. By approximating each $F_i$ by a Lipschitz function of norm at most $O_{\eps}(1)$ with an $L^1$-error of at most $\eps$, we may assume that $F_1,F_2,F_3$ have Lipschitz norm at most $M$ for some $M$. Let $A$ be sufficiently large in terms of $\eps,M$, let $N$ be sufficiently large in terms of $A$, and let $\theta \in (\R/\Z)^d$ be $(A,N)$-irrational. For $i \in \{1,2,3\}$, let $f_i: \Z\rightarrow [0,1]$ be the function defined by
$$
f_i(n) = F_i\Big(\frac{n}{N}, n\theta\Big).
$$
Since $\theta$ is $(A,N)$-irrational, we have $\sum_n f_i(n) = (\delta_i + O(\eps))N$ by \cite[Lemma A.4]{EberhardGreenManners} and
$$
T_{1,1,1}(f_1,f_2, f_3)  = (\wt{T}_{1,1,1}(F_1,F_2,F_3) + O(\eps))N^2
$$
by Lemma \ref{lem:T-Ttilde}. For $i\in \{1,2,3\}$, let $A_i \subset \Z$ be a random subset where each integer $n$ is chosen to be in $A_i$ with probability $f_i(n)$. Then $\E |A_i| = (\delta_i + O(\eps))N$, and moreover, we have $|A_i| = (\delta_i + O(\eps))N$ with probability $1-o(1)$ by Chebyshev's inequality. Since
$$
\E T_{1,1,1}(A_1,A_2,A_3) = T_{1,1,1}(f_1,f_2,f_3),
$$
it follows again by Chebyshev's inequality that there exists $A_1,A_2,A_3 \subset \Z$ with $|A_i| = (\delta_i + O(\eps))N$ such that
$$
T_{1,1,1}(A_1,A_2,A_3) \geq T_{1,1,1}(f_1,f_2,f_3) - \eps N^2.
$$
On the other hand, by a rearrangement inequality (see, for example, \cite[Theorem 1]{lev1998number}), $T_{1,1,1}(A_1,A_2,A_3)$ is maximized when each $A_i$ is a balanced interval around $0$, and hence
$$
T_{1,1,1}(A_1,A_2,A_3) \leq  (G(\delta_1,\delta_2,\delta_3) + O(\eps))N^2.
$$
Combining the estimates above concludes the proof.
\end{proof}

\begin{proof}[Proof of Proposition \ref{prop:proof3}]
It follows from Lemma \ref{lem:TF-TFa} that
$$
\wt{T}_{a_1,a_2,a_3}(F) = \frac{1}{|a_1a_2a_3|q^2} \sum_{\substack{r_1,r_2,r_3\Mod{q} \\ a_1r_1+a_2r_2+a_3r_3\equiv 0\Mod{q}}}  \wt{T}_{1,1,1}(F_{r_1}^{(a_1)}, F_{r_2}^{(a_2)}, F_{r_3}^{(a_3)}).
$$
Let $g:\Z/q\Z \rightarrow [0,1]$ be the function defined by $g(r) = (q\delta)^{-1}\int F_r$, so that $\sum_r g(r) = \delta^{-1}\E_r \int F_r = \delta^{-1}\int F = 1$. Then $\int F_{r_i}^{(a_i)} = |a_i| \int F_{r_i} = q\delta |a_i|g(r_i)$. It follows from Lemma \ref{lem:max-T111} that
$$
\wt{T}_{a_1,a_2,a_3}(F) \leq \frac{1}{|a_1a_2a_3|q^2} \sum_{\substack{r_1,r_2,r_3\Mod{q} \\ a_1r_1+a_2r_2+a_3r_3\equiv 0\Mod{q}}} G(q\delta |a_1|g(r_1), q\delta |a_2|g(r_2), q\delta |a_3|g(r_3)).
$$
Since the summand above is equal to $(q\delta)^2 G(|a_1|g(r_1), |a_2|g(r_2), |a_3|g(r_3))$, this concludes the proof of Proposition \ref{prop:proof3}.
\end{proof}

\begin{proof}[Proof of Lemma \ref{lem:q-divide-a1a2a3}]
Clearly $\gamma_{a_1,a_2,a_3}(q) \leq \gamma_{a_1,a_2,a_3}(qm)$. To prove $\gamma_{a_1,a_2,a_3}(q) \geq \gamma_{a_1,a_2,a_3}(qm)$, let $g: \Z/qm\Z \rightarrow [0,1]$ be an arbitrary function with $\sum_r g(r) = 1$, and we will prove that
$$
\gamma_{a_1,a_2,a_3}(g) \leq \gamma_{a_1,a_2,a_3}(\wt{g}),
$$
where $\wt{g}: \Z/q\Z \rightarrow [0,1]$ is defined by
$$
\wt{g}(r) = \sum_{s \in \Z/m\Z} g(r,s),
$$
where we identify $\Z/qm\Z$ with $\Z/q\Z \times \Z/m\Z$ in the natural way since $(q,m)=1$. By definition of $\gamma_{a_1,a_2,a_3}(g)$, we have
$$
\gamma_{a_1,a_2,a_3}(g) = \frac{1}{|a_1a_2a_3|} \sum_{\substack{r_1,r_2,r_3\Mod{q} \\ \sum_i a_ir_i\equiv 0\Mod{q}}} \sum_{\substack{s_1,s_2,s_3\Mod{m} \\ \sum_ia_is_i\equiv 0\Mod{m}}} G(|a_1|g(r_1,s_1), |a_2|g(r_2,s_2), |a_3|g(r_3,s_3)).
$$
Since $(m,a_1a_2a_3)=1$, after a change of variables $s_i'=a_is_i$ we may rewrite the inner sum over $s_1,s_2,s_3$ as
$$
\sum_{\substack{s_1',s_2',s_3'\Mod{m} \\ s_1'+s_2'+s_3'\equiv 0\Mod{m} }} G(|a_1|f_1(s_1'), |a_2|f_2(s_2'), |a_3|f_3(s_3')),
$$
where $f_i: \Z/m\Z \rightarrow [0,1]$ is defined by $f_i(s) = g(r_i, a_i^{-1}s)$. Since $\sum_{s} f_i(s) = \wt{g}(r_i)$, Lemma \ref{lem:G-sumbound} implies that the inner sum over $s_1,s_2,s_3$ is at most $G(|a_1|\wt{g}(r_1), |a_2|\wt{g}(r_2), |a_3|\wt{g}(r_3))$. This esbalishes the desired claim that $\gamma_{a_1,a_2,a_3}(g) \leq \gamma_{a_1,a_2,a_3}(\wt{g})$.
\end{proof}

\section{Computation of $\gamma_{1,1,3}(3)$ and $\gamma_{1,-1,3}(3)$}\label{sec:47/122}

In this section we prove Theorem \ref{thm:gamma113}. We start with some basic properties concerning the function $G(u,v,w)$.

\begin{lemma}\label{lem:G-convexity}
For all $u,v \geq 0$, the function $w\mapsto G(u,v,w)$ is concave on $w \in [0,+\infty)$. In particular, we have
$$
\sum_{i=1}^n G(u,v,w_i) \leq n G\Big(u,v,\frac{1}{n}\sum_{i=1}^nw_i\Big)
$$
for all $w_1,\cdots,w_n \geq 0$.
\end{lemma}

\begin{proof}
For fixed $u,v$, the derivative of the function $w\mapsto G(u,v,w)$ can be computed to be
$$
\frac{\partial}{\partial w} G(u,v,w) = \begin{cases} \min(u,v) & \text{if }w \leq |u-v|, \\ \frac{1}{2}(u+v-w) & \text{if } |u-v| \leq w \leq u+v, \\ 0 & \text{if }w \geq u+v, \end{cases}
$$
which is a decreasing function.
\end{proof}

\begin{lemma}\label{lem:G-uv}
For all $u,v,w \geq 0$ we have 
$$
G(u,v,w) \leq G\Big(\frac{u+v}{2}, \frac{u+v}{2}, w\Big).
$$
\end{lemma}

\begin{proof}
If $u+v \leq w$, then 
$$
G(u,v,w) = uv, \ \ G\Big(\frac{u+v}{2}, \frac{u+v}{2}, w\Big) = \frac{1}{4}(u+v)^2,
$$
and hence the conclusion holds. Now let $s = (u+v)/2$ and assume that $2s \geq w$. Without loss of generality, assume that $u \leq s$. For fixed $s,w$, the function $u\mapsto G(u, 2s-u, w)$ for $u \in [0,s]$ is defined by
$$
G(u, 2s-u, w) = \begin{cases} uw & \text{if }0 \leq u \leq s-\frac{w}{2}, \\ -(u-s)^2 + sw-\frac{w^2}{4} & \text{if }s-\frac{w}{2} \leq u \leq s. \end{cases}
$$
This function is clearly maximized at $u = s$.
\end{proof}

\subsection{Evaluating $\gamma_{1,1,\pm 3}$}

Let $g: \Z/3\Z\rightarrow [0,1]$ be a function with $g(0)+g(1)+g(2)=1$. Then
$$
\gamma_{1,1,\pm 3}(g) = \frac{1}{3} \sum_{x,z\Mod{3}} G(g(x), g(-x), 3g(z)).
$$
Let $u = g(0)$, $v = g(1)$, $w = g(2)$. The term with $x=z=0$ is $G(u,u,3u) = u^2$ and each of the four terms with $x,z \neq 0$ is $G(v,w,3g(z)) \leq vw$.  Hence
$$
\gamma_{1,1,\pm 3}(g) \leq \frac{1}{3}(u^2 + 4vw + 2G(v,w,3u) + G(u,u,3v) + G(u,u,3w)).
$$
By applying the AM-GM inequality to $4vw$, Lemma \ref{lem:G-uv} to $G(v,w,3u)$, and Lemma \ref{lem:G-convexity} to $G(u,u,3v) + G(u,u,3w)$, we obtain
\begin{align*}
\gamma_{1,1,\pm 3}(g) &\leq   \frac{1}{3}\Big(u^2 + (v+w)^2 + 2G\Big(\frac{v+w}{2},\frac{v+w}{2},3u\Big) + 2G\Big(u,u,\frac{3(v+w)}{2}\Big)\Big) \\
&= \frac{1}{3}\Big(u^2 + (1-u)^2 + 2G\Big(\frac{1-u}{2},\frac{1-u}{2},3u\Big) + 2G\Big(u,u,\frac{3(1-u)}{2}\Big)\Big).
\end{align*}
Let $f(u)$ be the right-hand side above. Divide into cases according to whether $1-u < 3u$ (or $u > 1/4$) and whether $2u < 3(1-u)/2$ (or $u < 3/7$). In each of the three regions $u \leq 1/4$, $1/4 \leq u \leq 3/7$, and $u \geq 3/7$, using the formulas for $G$ one can explicitly compute $f(u)$ as a quadratic function. This routine process leads to the conclusion that $f(u)$ is maximized at $u = 9/13$ and $f(9/13) = 5/13$.

\subsection{Evaluating $\gamma_{1,-1,\pm 3}$}

Let $g: \Z/3\Z\rightarrow [0,1]$ be a function with $g(0)+g(1)+g(2)=1$. Then
$$
\gamma_{1,-1,\pm 3}(g) = \frac{1}{3} \sum_{x,z\Mod{3}} G(g(x), g(x), 3g(z)).
$$
Let $u = g(0)$, $v = g(1)$, $w = g(2)$. Note that the right-hand side above is symmetric in $u,v,w$. Without loss of generality, assume that $u \leq v \leq w$. Using the formula
$$
G(\alpha,\alpha,\beta) = \alpha^2 - \Big(\alpha - \frac{\beta}{2}\Big)_+^2,
$$
where $\alpha_+= \max(\alpha,0)$, we obtain
$$
\gamma_{1,-1,\pm 3}(g) = u^2+v^2+w^2 - \frac{1}{3}\Big[\Big(v - \frac{3}{2}u\Big)_+^2 + \Big(w - \frac{3}{2}v\Big)_+^2 + \Big(w - \frac{3}{2}u\Big)_+^2\Big].
$$

\subsection*{Case 1: $v \geq (3/2)u$ and $w \geq (3/2)v$}

In this case, we have
$$
\gamma_{1,-1,\pm 3}(g) =  u^2+v^2+w^2 - \frac{1}{3}\Big[\Big(v - \frac{3}{2}u\Big)^2 + \Big(w - \frac{3}{2}v\Big)^2 + \Big(w - \frac{3}{2}u\Big)^2\Big].
$$
Direct computations show that
$$
\gamma_{1,-1,\pm 3}(g) = \frac{47}{122}(u+v+w)^2 - \frac{1}{366}\Big[18u - \frac{7}{3}(v+w)\Big]^2 - \frac{1}{27}\Big(w - \frac{7}{2}v\Big)^2.
$$
Since $u+v+w=1$, this is maximized when $w = (7/2)v$ and $18u = (7/3)(v+w)$, or when $(u,v,w) = (7/61, 12/61, 42/61)$, and we have $\gamma_{1,-1,\pm 3}(g) \leq 47/122 \approx 0.3852$.

\subsection*{Case 2: $v \geq (3/2)u$ and $w \leq (3/2)v$}

In this case, we have $v \geq (2/5)(1-u)$, $w \leq (3/5)(1-u)$, and
$$
\gamma_{1,-1,\pm 3}(g) = u^2+v^2+w^2 -\frac{1}{3}\Big[\Big(v - \frac{3}{2}u\Big)^2 + \Big(w - \frac{3}{2}u\Big)^2\Big] = u-\frac{3}{2}u^2 + \frac{2}{3}(v^2+w^2).
$$
Since
$$
2(v^2+w^2) = (v+w)^2 + (v-w)^2 \leq (1-u)^2 + \frac{1}{25}(1-u)^2 = \frac{26}{25}(1-u)^2,
$$
it follows that
$$
\gamma_{1,-1,\pm 3}(g) \leq u-\frac{3}{2}u^2 + \frac{26}{75}(1-u)^2 = -\frac{173}{150}\Big(u-\frac{23}{173}\Big)^2 + \frac{127}{346} \leq \frac{127}{346} \approx 0.367.
$$

\subsection*{Case 3: $v \leq (3/2)u$ and $w \geq (3/2)v$}

In this case, we have $u \geq (2/5)(1-w)$, $v \leq (3/5)(1-w)$, and
$$
\gamma_{1,-1,\pm 3}(g) =  u^2+v^2+w^2 - \frac{1}{3}\Big[\Big(w - \frac{3}{2}v\Big)^2 + \Big(w - \frac{3}{2}u\Big)^2\Big] = w-\frac{2}{3}w^2 + \frac{1}{4}(u^2+v^2).
$$
Since
$$
2(u^2+v^2) = (u+v)^2 + (u-v)^2 \leq (1-w)^2 + \frac{1}{25}(1-w)^2 = \frac{26}{25}(1-w)^2,
$$
it follows that
$$
\gamma_{1,-1,\pm 3}(g) \leq w -\frac{2}{3}w^2 + \frac{13}{100}(1-w)^2 = -\frac{161}{300}\Big(w - \frac{111}{161}\Big)^2 + \frac{62}{161} \leq \frac{62}{161} \approx 0.3851.
$$

\subsection*{Case 4: $v < (3/2)u$, $w <  (3/2)v$, and $w \geq (3/2)u$}

In this case, we have
$$
\gamma_{1,-1,\pm 3}(g) =  u^2+v^2+w^2 - \frac{1}{3}\Big(w - \frac{3}{2}u\Big)^2 = \frac{1}{4}u^2 + v^2 + \frac{2}{3}w^2 + uw.
$$
Let $F(u,v,w)$ be the function on the right-hand side above. Consider $F(u-\eps, v+\eps, w)$ for some sufficiently small $\eps > 0$. By Taylor expansion we have
$$
F(u-\eps, v+\eps, w) = F(u,v,w) + \Big(\frac{\partial F}{\partial v}(u,v,w) - \frac{\partial F}{\partial u}(u,v,w)\Big) \eps + O(\eps^2).
$$
Since
$$
\frac{\partial F}{\partial u}(u,v,w) = \frac{1}{2}u + w < \frac{1}{2}v + \frac{3}{2}v = 2v = \frac{\partial F}{\partial v}(u,v,w) = 2v,
$$
it follows that $F(u-\eps, v+\eps, w) > F(u,v,w)$ for $\eps > 0$ sufficiently small. This shows that $F(u,v,w)$ is maximized  when $v = (3/2)u$. Similarly, by considering $F(u, v-\eps, w+\eps)$ and noting that
$$
 \frac{\partial F}{\partial v}(u,v,w) = 2v < \frac{4}{3}w + u =  \frac{\partial F}{\partial w}(u,v,w),
$$
we may conclude that $F(u,v,w)$ is maximized when $w = (3/2)v$. This puts us in a situation already included in Case 1.

\subsection*{Case 5: $w \leq (3/2)u$}

In this case, we have $2/7 \leq v \leq 3/8$ and
$$
\gamma_{1,-1,\pm 3}(g) =  u^2+v^2+w^2 = v^2 + \frac{1}{2}(u+w)^2 + \frac{1}{2}(u-w)^2
$$
Since $w-u \leq (1/5)(u+w)$, we have
$$
\gamma_{1,-1,\pm 3}(g) \leq  v^2 + \frac{1}{2}(1-v)^2 + \frac{1}{50}(1-v)^2.
$$
On $v \in [2/7, 3/8]$, the right-hand side above is maximized at either of the two endpoints. After computations we see that it is maximized at $v = 2/7$ and $\gamma_{1,-1,\pm 3}(g) \leq 17/49 \approx 0.347$.

\section*{Acknowledgement}
Part of the work was done when M.W.X and S.Z
were members at SL Math 2025 Spring Extremal Combinatorics Program. We thank Jacob Fox and G\'abor Tardos for their interest. X.S. is supported by NSF grant DMS-2452462, M.W.X was supported by a Simons Junior Fellowship from the Simons Foundation.

\appendix	
	\section{Proof of Lemmas}

\begin{proof}[Proof of Lemma \ref{lemma ABAcupB}]
Write
\[
a:=|A|,\qquad b:=|B|.
\]
Since every element of $A$ is congruent to $1 \pmod 3$ and every element of $B$ is congruent to $2 \pmod 3$, every element of $3(A\cup B)$ is congruent to $3$ or $6 \pmod 9$. In particular,
\begin{equation}
\label{eq:no9}
3(A\cup B)\cap 9\mathbb Z=\varnothing.
\end{equation}

For $r\in\{1,4,7\}$ and $s\in\{2,5,8\}$, define
\[
A^{(r)}:=\{x\in A:x\equiv r \pmod 9\},\qquad
B^{(s)}:=\{y\in B:y\equiv s \pmod 9\}.
\]
Thus $\{A^{(1)},A^{(4)},A^{(7)}\}$ partitions $A$ and $\{B^{(2)},B^{(5)},B^{(8)}\}$ partitions $B$.

We say that a residue class is \emph{popular} in $A$ if it contains at least $\frac45a$ elements of $A$, and similarly for $B$.

\medskip
\noindent\textbf{Case 1: neither $A$ nor $B$ has a popular residue class.}

Then
\[
|A^{(r)}|<\frac45a\quad (r\in\{1,4,7\}),\qquad
|B^{(s)}|<\frac45b\quad (s\in\{2,5,8\}).
\]
Let $D$ denote the number of pairs $(x,y)\in A\times B$ such that $x+y\equiv 0\pmod 9$. It is clear that
\[
D=|A^{(1)}||B^{(8)}|+|A^{(4)}||B^{(5)}|+|A^{(7)}||B^{(2)}|.
\]
Every such pair satisfies $x+y\in 9\mathbb Z$, so by \eqref{eq:no9} it contributes nothing to $T(A,B,3(A\cup B))$. Hence
\begin{equation}
\label{eq:TleabminusD}
T(A,B,3(A\cup B))\le ab-D.
\end{equation}

We claim that
\[
D\ge \frac1{25}ab.
\]
Indeed, set
\[
\alpha_1:=\frac{|A^{(1)}|}{a},\quad
\alpha_2:=\frac{|A^{(4)}|}{a},\quad
\alpha_3:=\frac{|A^{(7)}|}{a},
\]
and
\[
\beta_1:=\frac{|B^{(8)}|}{b},\quad
\beta_2:=\frac{|B^{(5)}|}{b},\quad
\beta_3:=\frac{|B^{(2)}|}{b}.
\]
Then $\alpha_i,\beta_i\ge 0$, $\alpha_1+\alpha_2+\alpha_3=1$, $\beta_1+\beta_2+\beta_3=1$, and $\alpha_i,\beta_i\le \frac45$ for all $i$. Also
\[
\frac{D}{ab}=\alpha_1\beta_1+\alpha_2\beta_2+\alpha_3\beta_3.
\]
After relabelling the indices, we may assume
\[
\beta_1\le \beta_2\le \beta_3.
\]
For fixed $(\beta_1,\beta_2,\beta_3)$, the minimum of $\alpha_1\beta_1+\alpha_2\beta_2+\alpha_3\beta_3$ subject to
\[
\alpha_i\ge 0,\qquad \alpha_1+\alpha_2+\alpha_3=1,\qquad \alpha_i\le \frac45
\]
is attained at
\[
\alpha_1=\frac45,\qquad \alpha_2=\frac15,\qquad \alpha_3=0.
\]
Therefore
\[
\frac{D}{ab}\ge \frac45\beta_1+\frac15\beta_2\ge \frac15(\beta_1+\beta_2).
\]
Since $\beta_3\le \frac45$, we have $\beta_1+\beta_2=1-\beta_3\ge \frac15$, and so
\[
\frac{D}{ab}\ge \frac1{25}.
\]
This proves the claim.

Combining the claim with \eqref{eq:TleabminusD}, we obtain
\[
T(A,B,3(A\cup B))\le ab-\frac1{25}ab=\frac{24}{25}ab
\le \frac{24}{25}\cdot \frac{(a+b)^2}{4}
=\frac6{25}(a+b)^2.
\]

\medskip
\noindent\textbf{Case 2: at least one of $A$ and $B$ has a popular residue class.}

By multiplying all elements of $A\cup B$ with $2$ if needed, we may assume that $A$ has a popular residue class. Choose $r_0\in\{1,4,7\}$ such that
\[
|A^{(r_0)}|\ge \frac45a.
\]
Set
\[
A^{\star}:=A^{(r_0)},\qquad A^{\circ}:=A\setminus A^{\star},
\qquad a_{\star}:=|A^{\star}|,\qquad a_{\circ}:=|A^{\circ}|.
\]
Then
\begin{equation}
\label{eq:astar}
a_{\star}\ge \frac45a,\qquad a_{\circ}\le \frac15a.
\end{equation}

Let $s_0\in\{2,5,8\}$ be the unique residue class such that
\[
r_0+s_0\equiv 3\pmod 9.
\]
Set
\[
B^{\star}:=B^{(s_0)},\qquad B^{\circ}:=B\setminus B^{\star},
\qquad b_{\star}:=|B^{\star}|,\qquad b_{\circ}:=|B^{\circ}|.
\]

If $x\in A^{\star}$ and $y\in B^{\star}$, then $x+y\equiv 3\pmod 9$, hence any solution
\[
x+y=z\in 3(A\cup B)
\]
must satisfy $z\in 3A$. Likewise, if $x\in A^{\star}$ and $y\in B^{\circ}$, then $x+y\not\equiv 3\pmod 9$, so any solution with $z\in 3(A\cup B)$ must satisfy $z\in 3B$ (if $x+y\equiv 0\pmod 9$, then there is no solution by \eqref{eq:no9}). Therefore
\begin{equation}
\label{eq:decomp}
T(A,B,3(A\cup B))
=
T(A^{\star},B^{\star},3A)
+
T(A^{\star},B^{\circ},3B)
+
T(A^{\circ},B,3(A\cup B)).
\end{equation}
Also
\begin{equation}
\label{eq:lasttrivial}
T(A^{\circ},B,3(A\cup B))\le a_{\circ}b.
\end{equation}

We now distinguish subcases.

\medskip
\noindent\textbf{Subcase 2a: $b_{\star}>a_{\star}+a$.}

Then
\[
b\ge b_{\star}>a_{\star}+a\ge \frac45a+a=\frac95a.
\]
Hence
\[
T(A,B,3(A\cup B))\le ab\le \frac6{25}(a+b)^2,
\]
as $b\ge \frac95a\ge\frac32a$.

\medskip
\noindent\textbf{Subcase 2b: $b_{\star}\le a_{\star}+a$ and $a_{\star}\le b_{\circ}+b$.}

Then Lemma~\ref{lemma ABC} applies to both $T(A^{\star},B^{\star},3A)$ and $T(A^{\star},B^{\circ},3B)$. Thus
\[
T(A^{\star},B^{\star},3A)
\le
a_{\star}b_{\star}
-\frac14\max(a_{\star}+b_{\star}-a,0)^2
+1,
\]
and
\[
T(A^{\star},B^{\circ},3B)
\le
a_{\star}b_{\circ}
-\frac14\max(a_{\star}+b_{\circ}-b,0)^2
+1.
\]
Combining these with \eqref{eq:decomp} and \eqref{eq:lasttrivial},
\[
T(A,B,3(A\cup B))
\le
ab
-\frac14(u^2+v^2)
+2,
\]
where
\[
u:=\max(a_{\star}+b_{\star}-a,0),\qquad
v:=\max(a_{\star}+b_{\circ}-b,0).
\]
Since $b_{\star}+b_{\circ}=b$, we have
\[
(a_{\star}+b_{\star}-a)+(a_{\star}+b_{\circ}-b)=2a_{\star}-a.
\]
Hence
\[
u+v
\ge
\max\bigl((a_{\star}+b_{\star}-a)+(a_{\star}+b_{\circ}-b),\,0\bigr)
=
2a_{\star}-a.
\]
Using \eqref{eq:astar}, this gives
\[
u+v\ge 2a_{\star}-a\ge \frac35a.
\]
Therefore
\[
u^2+v^2\ge \frac12(u+v)^2\ge \frac12\left(\frac35a\right)^2=\frac{9}{50}a^2,
\]
and so
\[
T(A,B,3(A\cup B))
\le
ab-\frac14\cdot \frac{9}{50}a^2+2
=
ab-\frac{9}{200}a^2+2.
\]
Finally, it is easy to check that 
\[
ab-\frac{9}{200}a^2\le \frac6{25}(a+b)^2
.\]
Thus
\[
T(A,B,3(A\cup B))\le \frac6{25}(a+b)^2+2.
\]

\medskip
\noindent\textbf{Subcase 2c: $b_{\star}\le a_{\star}+a$ and $a_{\star}>b_{\circ}+b$.}

In this subcase Lemma~\ref{lemma ABC} still applies to $T(A^{\star},B^{\star},3A)$, so
\[
T(A^{\star},B^{\star},3A)
\le
a_{\star}b_{\star}
-\frac14\max(a_{\star}+b_{\star}-a,0)^2
+1.
\]
For the second term we use the trivial bound
\[
T(A^{\star},B^{\circ},3B)\le b_{\circ}b.
\]
Together with \eqref{eq:decomp} and \eqref{eq:lasttrivial}, this gives
\[
T(A,B,3(A\cup B))
\le
a_{\circ}b
+
a_{\star}b_{\star}
-\frac14\max(a_{\star}+b_{\star}-a,0)^2
+1
+
b_{\circ}b.
\]
Since $b_{\star}=b-b_{\circ}$ and $a_{\star} + b_{\star} - a = b_{\star} - a_{\circ}$, we obtain
\[
T(A,B,3(A\cup B))
\le
ab
-
(a_{\star}-b)b_{\circ}
-
\frac14\max(b_{\star}-a_{\circ},0)^2
+1.
\]
Discarding the nonnegative term $(a_{\star}-b)b_{\circ}$, we get
\begin{equation}
\label{eq:subcase2c-main}
T(A,B,3(A\cup B))
\le
ab
-
\frac14\max(b_{\star}-a_{\circ},0)^2
+1.
\end{equation}

Now $a_{\star}>b_{\circ}+b$ implies
\[
b_{\star}-a_{\circ}
=
b-b_{\circ}-a+a_{\star}
>
2b-a.
\]
It follows that
\[
\max(b_{\star}-a_{\circ},0)\ge \max(2b-a,0).
\]
So \eqref{eq:subcase2c-main} yields
\[
T(A,B,3(A\cup B))
\le
ab-\frac14\max(2b-a,0)^2+1.
\]
We claim that
\[
ab-\frac14\max(2b-a,0)^2\le \frac6{25}(a+b)^2.
\]
Indeed, setting $x = a / b$, this is equivalent to the elementary inequality
$$x - \frac{1}{4}\max\left(2 - x, 0\right)^{2} \leq \frac{6}{25}\left(x + 1\right)^{2}$$
which holds for any $x \in \mathbb{R}$. This completes the proof.
\end{proof}

\begin{proof}[Proof of Lemma~\ref{lem 3.4}]
We write
$$
n:=|S|,\qquad c:=|S_0|,\qquad s:=|S_1|+|S_2|=n-c.
$$
Recall that

$$T(S, S, 3 \cdot S) = 2 T(S_1, S_2, 3 \cdot S) + T(S_0, S_0, 3 \cdot (S_1 \cup S_2)) + T(S_0, S_0, 3 \cdot S_0).
$$
Bounding the third term using \cref{lem:half-bound}, we get
\begin{equation}
\label{eq:full}T(S, S, 3 \cdot S) \leq 2 T(S_1, S_2, 3 \cdot S) + T(S_0, S_0, 3 \cdot (S_1 \cup S_2)) + \frac{1}{2} |S_0|^2 + 1.
\end{equation}
We bound the middle term by Lemma~\ref{lemma ABC} with $A=B=S_0$ and $C=3(S_1\cup S_2)$ (so $|C|=|S_1\cup S_2|=s$):
\begin{equation}
\label{eq:middle}
T(S_0,S_0,3 \cdot (S_1\cup S_2))
\le c^2-\frac14\max(2c-s,0)^2+1
=c^2-\frac14\max(3c-n,0)^2+1.
\end{equation}

For the first term we have two bounds.
First, we have the trivial bound
\begin{equation}
\label{eq:trivial-first}
T(S_1,S_2,3S)\le |S_1||S_2| \leq \frac{s^2}{4}.
\end{equation}
Second, since $3S=3(S_1\cup S_2)\,\dot\cup\, 3S_0$, we have
\[
T(S_1,S_2,3S)=T(S_1,S_2,3(S_1\cup S_2))+T(S_1,S_2,3S_0).
\]
By Lemma~\ref{lemma ABAcupB},
\[
T(S_1,S_2,3(S_1\cup S_2))\le \frac{6}{25}s^2+2.
\]
Moreover, for each fixed $z\in 3S_0$ the number of pairs $(x,y)\in S_1\times S_2$ with $x+y=z$ is at most
$\min(|S_1|,|S_2|)\le \frac{s}{2}$, hence
\[
T(S_1,S_2,3S_0)\le |3S_0|\cdot \frac{s}{2}=c\cdot \frac{s}{2}.
\]
Therefore
\begin{equation}
\label{eq:aba-first}
2T(S_1,S_2,3S)\le \frac{12}{25}s^2+cs+4.
\end{equation}

We now split according to the size of $s$.

\medskip
\noindent\textbf{Case 1: $s\le \frac{50}{51}n$ (equivalently $c\ge \frac{n}{51}$).}
Use \eqref{eq:trivial-first} for the first term and \eqref{eq:middle} for the middle term in \eqref{eq:full} to get
\[
T(S,S,3S)\le \frac{s^2}{2}+\left(c^2-\frac14\max(3c-n,0)^2+1\right)+\frac12c^2+1
= \frac{s^2}{2}+\frac32c^2-\frac14\max(3c-n,0)^2+2.
\]
Write $x:=c/n\in\left[\frac{1}{51},\frac{2}{3}\right]$ and note $s=(1-x)n$.
If $x\le \frac13$, then $\max(3x-1,0)=0$ and
\[
\frac{1}{n^2}T(S,S,3S)\le \frac12(1-x)^2+\frac32x^2+\frac{2}{n^2}
=\frac12-x+2x^2+\frac{2}{n^2},
\]
which is maximized at an endpoint and it follows that  $\frac{1}{n^2}T(S,S,3S)\le \frac{2503}{5202}+\frac{2}{n^2}$.
If instead $x\ge \frac13$, then $\max(3x-1,0)=3x-1$ and a direct simplification yields
\[
\frac12(1-x)^2+\frac32x^2-\frac14(3x-1)^2
=\frac12-\frac14(1-x)^2,
\]
which is increasing in $x$ on $[1/3,2/3]$, so its maximum there is at $x=2/3$ and equals $\frac{17}{36}<\frac{2503}{5202}$.
Thus in Case~1,
\begin{equation}
\label{eq:case1}
T(S,S,3S)\le \frac{2503}{5202}n^2+2.
\end{equation}

\medskip
\noindent\textbf{Case 2: $s>\frac{50}{51}n$ (equivalently $c<\frac{n}{51}$).}
Here $x=c/n<1/51<1/3$, so $\max(3c-n,0)=0$ and \eqref{eq:middle} gives
$T(S_0,S_0,3(S_1\cup S_2))\le c^2+1$.
Using \eqref{eq:aba-first} and \eqref{eq:full} we obtain
\[
T(S,S,3S)\le \left(\frac{12}{25}s^2+cs+4\right)+(c^2+1)+\frac12c^2+1
= \frac{12}{25}s^2+cs+\frac32c^2+6.
\]
Again $s=(1-x)n$ and dividing by $n^2$ the quadratic becomes
\[
\frac{12}{25}(1-x)^2+x(1-x)+\frac32x^2
=\frac{12}{25}+\frac{1}{25}x+\frac{49}{50}x^2,
\]
which is increasing for $x\ge 0$; hence on $[0,1/51]$ it is maximized at $x=1/51$ which leads to that $\le \frac{2503}{5202}n^2+6.$

Combining with \eqref{eq:case1}, we obtain for all winnable $S$ that
\[
T(S,S,3S)\le \frac{2503}{5202}n^2+6 = \left(\frac12-\frac{49}{2601}\right)|S|^2+6,
\]
as claimed (take $\varepsilon=\frac{49}{2601}$).
\end{proof}

\bibliographystyle{plain}
\bibliography{solution}{}

@article{Aa,
  title={Maximising the number of solutions to a linear equation in a set of integers},
  author={Aaronson, James},
  journal={Bulletin of the London Mathematical Society},
  volume={51},
  number={2},
  pages={311--326},
  year={2019},
  publisher={Wiley Online Library}
}

@article{lev1998number,
  title={On the number of solutions of a linear equation over finite sets},
  author={Lev, Vsevolod F.},
  journal={Journal of Combinatorial Theory, Series A},
  volume={83},
  number={2},
  pages={251--267},
  year={1998},
  doi={10.1006/jcta.1998.2878}
}

@book{hardy1952inequalities,
  title={Inequalities},
  author={Hardy, G. H. and Littlewood, J. E. and P{\'o}lya, G.},
  year={1952},
  publisher={Cambridge University Press}}

@article{hardy1928notes,
  title={Notes on the theory of series (VIII): An inequality},
  author={Hardy, G. H. and Littlewood, J. E.},
  journal={Journal of the London Mathematical Society},
  volume={1},
  number={3},
  pages={105--110},
  year={1928}
}

@article{gabriel1931rearrangement,
  title={The rearrangement of positive Fourier coefficients},
  author={Gabriel, R. M.},
  journal={Proceedings of the London Mathematical Society},
  volume={2},
  number={33},
  pages={32--51},
  year={1931}
}

@article{green2008maximal,
  title={On the maximal number of 3-term arithmetic progressions in subsets of $\mathbb{Z}/N\mathbb{Z}$},
  author={Green, Ben and Sisask, Olof},
  journal={Bulletin of the London Mathematical Society},
  volume={40},
  number={6},
  pages={945--955},
  year={2008},
  doi={10.1112/blms/bdn074}
}

@article{lev2014solving,
  title   = {Solving $a \pm b = 2c$ in elements of finite sets},
  author  = {Lev, Vsevolod F. and Pinchasi, Rom},
  journal = {Acta Arithmetica},
  volume  = {163},
  number  = {2},
  pages   = {127--140},
  year    = {2014},
  doi     = {10.4064/aa163-2-2}
}

@article {GreenRuzsa,
    AUTHOR = {Green, B. and Ruzsa, I. Z.},
     TITLE = {Sets with small sumset and rectification},
   JOURNAL = {Bull. London Math. Soc.},
  FJOURNAL = {The Bulletin of the London Mathematical Society},
    VOLUME = {38},
      YEAR = {2006},
    NUMBER = {1},
     PAGES = {43--52},
      ISSN = {0024-6093,1469-2120},
   MRCLASS = {11B75},
  MRNUMBER = {2201602},
MRREVIEWER = {Mei\ Chu\ Chang},
       DOI = {10.1112/S0024609305018102},
       URL = {https://doi-org.ezproxy.uky.edu/10.1112/S0024609305018102},
}

@article{Eberhard-note,
      title={The abelian arithmetic regularity lemma}, 
      author={Eberhard, S.},
      year={2016},
      note={arXiv:1606.09303},
}

@incollection {GreenTao-regularity,
    AUTHOR = {Green, B. and Tao, T.},
     TITLE = {An arithmetic regularity lemma, an associated counting lemma,
              and applications},
 BOOKTITLE = {An irregular mind},
    SERIES = {Bolyai Soc. Math. Stud.},
    VOLUME = {21},
     PAGES = {261--334},
 PUBLISHER = {J\'anos Bolyai Math. Soc., Budapest},
      YEAR = {2010},
      ISBN = {978-963-9453-14-2; 978-3-642-14443-1},
   MRCLASS = {11B30 (05D05)},
  MRNUMBER = {2815606},
MRREVIEWER = {David\ Conlon},
       DOI = {10.1007/978-3-642-14444-8\_7},
       URL = {https://doi-org.ezproxy.uky.edu/10.1007/978-3-642-14444-8_7},
}

@article {EberhardGreenManners,
    AUTHOR = {Eberhard, S. and Green, B. and Manners, F.},
     TITLE = {Sets of integers with no large sum-free subset},
   JOURNAL = {Ann. of Math. (2)},
  FJOURNAL = {Annals of Mathematics. Second Series},
    VOLUME = {180},
      YEAR = {2014},
    NUMBER = {2},
     PAGES = {621--652},
      ISSN = {0003-486X,1939-8980},
   MRCLASS = {11B30},
  MRNUMBER = {3224720},
MRREVIEWER = {Olof\ Sisask},
       DOI = {10.4007/annals.2014.180.2.5},
       URL = {https://doi-org.ezproxy.uky.edu/10.4007/annals.2014.180.2.5},
}

@book {TaoVu-book,
    AUTHOR = {Tao, T. and Vu, V.},
     TITLE = {Additive combinatorics},
    SERIES = {Cambridge Studies in Advanced Mathematics},
    VOLUME = {105},
 PUBLISHER = {Cambridge University Press, Cambridge},
      YEAR = {2006},
     PAGES = {xviii+512},
      ISBN = {978-0-521-85386-6; 0-521-85386-9},
   MRCLASS = {11-02 (05-02 05D10 11B13 11P70 11P82 28D05 37A45)},
  MRNUMBER = {2289012},
MRREVIEWER = {Serge\u i\ V.\ Konyagin and Ilya\ D.\ Shkredov},
       DOI = {10.1017/CBO9780511755149},
       URL = {https://doi-org.ezproxy.uky.edu/10.1017/CBO9780511755149},
}

@misc{Korsky,
  author       = {Korsky, S.},
  title        = {Affine copies of three-point patterns in sets of integers},
  year         = {2026},
  eprint       = {2609.02308},
  archivePrefix = {arXiv},
  primaryClass = {math.CO},
  url          = {https://arxiv.org/abs/2609.02308}
}
\end{document}